\documentclass[11pt,twoside,reqno]{amsart}
\usepackage{amsmath}
\usepackage{amstext}
\usepackage{amssymb}
\usepackage{epsfig}
\usepackage{graphicx}
\usepackage{mathrsfs}
\usepackage{cancel}
\usepackage[parfill]{parskip}
\usepackage[german]{varioref}
\usepackage[all]{xy}
\usepackage{color}
\allowdisplaybreaks
\calclayout
\makeatletter
\renewcommand{\@seccntformat}[1]{\bf\csname the#1\endcsname.}
\renewcommand{\section}{\@startsection{section}{1}
  \z@{.7\linespacing\@plus\linespacing}{.5\linespacing}
  {\normalfont\upshape\bfseries\centering}}
\renewcommand{\@biblabel}[1]{\@ifnotempty{#1}{#1.}}
\makeatother

\theoremstyle{plain}
\newtheorem{thm}{Theorem}[section]

\newtheorem{proposition}[thm]{Proposition}
\newtheorem{cor}[thm]{Corollary}

\theoremstyle{definition}
\newtheorem{ex}[thm]{Example}
\newtheorem{definition}[thm]{Definition}

\def\>{\succ}
\def\<{\prec}

\def\b{\beta}

\begin{document}

\title[Splitting of operations for Hom-diassociative and Hom-triassociative Algebras]{Splitting of operations for Hom-diassociative Algebras and Hom-triassociative Algebras}

\author{Abdelkader Hamdouni\textsuperscript{1}, Imed Basdouri \textsuperscript{2}, Mariem Jendoubi\textsuperscript{3}, Ahmed Zahari Abdou\textsuperscript{4}}

\address{\textsuperscript{1} University of Carthage· Higher Institute of Sciences and Technology of Environment of Borj Cedria (ISTEUB),  Tunisia}
\address{\textsuperscript{2} Universit\'e de Gafsa, Facultté des sciences de Gafsa, Tunis}
\address{\textsuperscript{3}D\'epartement de Math\'ematiques, Facult\'e des Sciences, Universit\'e de Sfax, Sfax, Tunisia}
\address{\textsuperscript{4}IRIMAS - D\'epartement de Math\'ematiques, Facult\'e des Sciences, Universit\'e de Haute Alsace, Mulhouse, France}

\email{\textsuperscript{1}abdelkader.hamdouni@isimg.rnu.tn}
\email{\textsuperscript{2}basdourimed@yahoo.fr}
\email{\textsuperscript{3}mariemj090@gmail.com}
\email{\textsuperscript{4}abdou-damdji.ahmed-zahari@uha.fr}


\date{\today}

	%\thanks{This work was supported by the National Natural Science Foundation of  (No. 11871421  11501515) and the Zhejiang Provincial Natural Science Foundation of China (No. LQ16A010011).}
	%%%%%%%%%%%%%%%%%%%%%%%%%%%%%%%%%%%%%%%%%%%%%%%%%%%%%%%%%%%%%%%%%%%%%%%%%
	%%%%%%%%%%%%%%%%%%%%%%%%%%%%%%%%%%%%%%%%%%%%%%%%%%%%%%%%%%%%%%%%%%%%%%%%%
	\begin{abstract}  
     Hom-quadri dendriform algebras and Hom-six-dendriform agebras are introduced and studied which is a splitting of a Hom-diassociative and Hom-triassociative algebras, respectively. Moreover we explore the connections be
tween these categories of Hom-algebras. Finally  We elaborate a classification of Hom-quadri-dendriform algebra in low dimensional.
  \end{abstract}

\maketitle 

\section{Introduction}\label{introduction}
The concept of diassociative algebras was first introduced by Loday in \cite{M2008}  inspired by his research on periodicity phenomena in algebraic K-theory and his study of Leibniz algebras.
In \cite{LR}  Loday and Ronco demonstrated that the family of chain modules over standard simplices can be endowed with an operad structure. The algebras over this operad are known as triassociative algebras  characterized by three binary operations satisfying eleven defining identities. In the same work  they introduced tridendriform algebras and established that triassociative algebras are Koszul dual to tridendriform algebras \cite{E,EMP,ZGM}.
Inspired by Loday’s work \cite{L} and Loday and Ronco \cite{LR}.  Wen Teng \cite{WenTeng} introduced new algebraic structures quadri-dendriform algebras and six-dendriform algebras  as natural splittings of diassociative and triassociative algebras  respectively.
Dendriform algebras  also introduced by Loday in \cite{Basdouri1,MBZ1,Basdouri2,LR,AM}  decompose the associative product into two binary operations such that their sum defines an associative operation. These algebras are intimately connected to diassociative algebras via Koszul duality  where the Koszul dual of the operad associated to diassociative algebras is precisely the operad of dendriform algebras.
Dendriform algebras appear in several key areas of mathematics including algebraic K-theory,  combinatorics (notably in the study of trees), cohomology  and deformation theory
In recent years  considerable attention has been devoted to Hom-type algebras  due in part to their potential to provide a unified framework for constructing various types of natural deformations  which are of interest in both mathematics and mathematical physics.
Hom-dendriform algebras were introduced in \cite{D0, D1, DS, MS,MAZ, Y} as twisted versions of dendriform algebras and as natural splittings of Hom-associative algebras. 
The notions of Hom-associative, Hom-diassociative  and Hom-triassociative algebras  which generalize classical associative structures by twisting the associativity condition through a linear map  were developed in \cite{D1,  MBZ, AZ, ZB, Zahari2025}.
Motivated by Loday’s  Ronco and Wen Teng \cite{L, LR, WenTeng}  We introduce new algebraic structure Hom-quadri-dendriform algebras and Hom-six-dendriform algebras.
 The purpose of this paper is to introduce the notion of Hom-quadri
dendriform algebras as a splitting of di-associative algebras and 
 the notion of Hom-six-dendriform algebras as a splitting of tri-associative algebras. use relative averaging operators on Hom- algebras  as a framework to construct and study these new Hom-algebraic structures.The idea of an averaging operator was first implicitly introduced by Reynolds in 1895. Throughout the twentieth century  these operators were mainly investigated within the algebra of functions defined on a given space. In his doctoral dissertation  In \cite{AO} Cao examined averaging operators from both algebraic and combinatorial viewpoints. More recent research has further explored averaging operators on any binary operad  highlighting their connections to concepts such as bisuccessors  duplicators  and Rota-Baxter operators in 
\cite{AG, AO1, AO2,Basdouri4, Basdouri3,MK}.  
The paper is structured as follows: In Section 1  we recall the basic definitions and properties of Hom-dendriform and Hom-diassociative algebras. In Section 2  we introduce the notion of Hom-quadri-dendriform algebras and define relative averaging operators on Hom-dendriform algebras. We show that such an operator gives rise to a Hom-quadri-dendriform structure  and conversely  that any such structure can be embedded into a Hom-dendriform algebra equipped with a relative averaging operator. In Section 3  we define Hom-six-dendriform algebras as splittings of Hom-triassociative algebras. We also demonstrate that homomorphic relative averaging operators yield Hom-six-dendriform structures. Finally  we present a classification of low-dimensional Hom-quadri-dendriform algebras  illustrating structural properties and examples.\\
 Throughout this paper  $K$ denotes a field of characteristic zero. All the vector spaces
 and (multi)linear maps are taken over $K$.
\section{Hom-dendriform algebra and Hom-diassociative algebra}

In this section  we introduce the notions of Hom-dendriform algebras and Hom-diassociative algebras
 generalizing the classical dendriform and diassociative algebras to Hom-algebras setting.
 \subsection{Hom-dendriform algebras}\ 
Dendriform algebras were introduced by Loday in \cite{L}. Dendriform algebras are algebras with two operations  which dichotomize the notion of associative algebra. We
 generalize now this notion by twisting the identities by a linear map.
\begin{definition}\label{d1}
A Hom-dendriform algebra is a quadruple $(D,\  \prec,\  \succ,\  \alpha)$  where $D$ is a $\mathbb{K}$-vector space  the operations $\prec, \  \succ : \ D \times D \to D$ are bilinear maps and $\alpha : \ D \to D$ is a linear map (called the structure map)  such that for all $x , \ y,  z \in D$  the following identities hold:
\begin{align}
\alpha(x) \prec (y \prec z + y \succ z)& = (x \prec y) \prec \alpha(z) \label{eq1} \\
\alpha(x) \succ (y \prec z)& = (x \succ y) \prec \alpha(z) \label{eq2} \\
\alpha(x) \succ (y \succ z)& = (x \prec y + x \succ y) \succ \alpha(z)  \label{eq3}.
\end{align}
$(D,\  \prec,\  \succ,\  \alpha)$ is said to be a multiplicative Hom-algebra if $\forall x,\  y \in D$ we have $$\alpha(x\prec y)=\alpha(x)\prec\alpha(y)\quad \text{and}\quad\alpha(x\succ y)=\alpha(x)\succ\alpha(y).$$
\end{definition}
\begin{definition}\label{d2}
Given two Hom-dendriform algebras $(D,\  \prec,\  \succ,\  \alpha)$ \ and $(D',\  \prec',\  \succ',\  \alpha')$. A Hom-dendriform homomorphism is a linear map $T : \ D \to D'$ that satisfies the following compatibility conditions\\
$T(x \prec y) =T(x) \prec' T(y) ,\quad 
T(x \succ y) = T(x) \succ'T(y)  \quad \ and  \   T \circ \alpha = \alpha' \circ T$ 
for all $x ,\ y\in D$.\\
 In particular  if $T$ is bijective then $(D,\  \prec,\  \succ,\  \alpha)$ \ and \ $(D',\  \prec',\  \succ',\  \alpha')$ are said to be isomorphic.
\end{definition}
\begin{ex}
Let $(D,\  \prec,\  \succ,\  \alpha)$ be a Hom-dendriform algebra with a basis $\left\{e_1 ,\ e_2 ,\ e_3\right\}$ whose non-zero product are given by
$$
D(\eta)= \left\{
    \begin{array}{ll}
      e_1\prec e_2=\eta e_3&e_2\succ  e_1=e_3\\
      e_2\prec e_2=-\frac{1}{2}e_3&e_2\succ  e_2=e_1+\frac{1}{4}e_3\\
			e_3\prec e_2=e_3&e_3\succ  e_2=\eta e_1+e_3 
    \end{array}
\right.,\quad
\alpha = 
\begin{pmatrix}
0 & 1 &0 \\
0 & 0 &0 \\
0 &0&b
\end{pmatrix},\quad  b,\ \eta \in\mathbb{R}.
$$
\end{ex}

\begin{definition}\label{d5}
Let $(D,\  \prec,\  \succ,\  \alpha)$ be a Hom-dendriform algebra. A representation of $ D$ is given
 by a vector space $M$ together with four bilinear maps
$\prec_l ,\ \succ_l :\ D \times M \to M  ,\quad
\prec_r ,\ \succ_r :\ M \times D \to M  $
    representing the left and right actions. These maps must satisfy the following nine compatibility conditions: $\forall \ x,\ y \in D   ,\ m \in M $ we have
\begin{align*}
&\textbf{(I) Left compatibility conditions} \\
(x \prec y) \prec_l \beta(m) &= \alpha(x) \prec_l (y \prec_l m + y \succ_l m), \\
(x \succ y) \prec_l \beta(m) &= \alpha(x) \succ_l (y \prec_l m), \\
\alpha(x) \succ_l (y \succ_l m) &= (x \prec y + x \succ y) \succ_l \beta(m). \\
&\textbf{(II) Right compatibility conditions} \\
\beta(m) \prec_r (x \prec y + x \succ y) &= (m \prec_r x) \prec_r \alpha(y), \\
\beta(m) \succ_r (x \prec y) &= (m \succ_r x) \prec_r \alpha(y), \\
(m \prec_r x + m \succ_r x) \succ_r \alpha(y) &= \beta(m) \succ_r (x \succ y).\\
&\textbf{(III) Mixed compatibility conditions} \\
(x \prec_l m) \prec_r \alpha(y) &= \alpha(x) \prec_l (m \prec_r y + m \succ_r y), \\
(x \succ_l m) \prec_r \alpha(y) &= \alpha(x) \succ_l (m \prec_r y), \\
(x \prec_l m + x \succ_l m) \succ_r \alpha(y) &= \alpha(x) \succ_l (m \succ_r y).
\end{align*}
Any Hom-dendriform algebra $(D,\  \prec,\  \succ,\  \alpha)$ is a representation of itself with
$\prec=\prec_{l}=\prec_r \ and \  \succ=\succ_l=\succ_r.$
 This representation is known as the adjoint representation.
\end{definition}
\subsection{ Hom-diassociative algebras}
Diassociative algebras were introduced by Jean-Louis Loday \cite{L} as algebras equipped with two associative multiplications that satisfy specific compatibility conditions. We generalize now this notion by twisting the identities by a linear map.
 \begin{definition}\label{d6}
A Hom-diassociative algebras is a $4$-truple $(D,\  \dashv,\  \vdash,\  \alpha)$ consisting of a  linear space $D$  bilinear maps
 $\dashv ,\ \vdash:\ D\times D \longrightarrow D$ and  $\alpha :\ D\longrightarrow D$ satisfying  for all $x  ,\ y,\  z\in D$ 
the following conditions: 
\begin{eqnarray}
(x\dashv y)\dashv\alpha(z)&=&\alpha(x)\dashv(y\dashv z) \label{eq4}\\
(x\dashv y)\dashv\alpha(z)&=&\alpha(x)\dashv(y\vdash z) \label{eq5}\\
(x\vdash y)\dashv\alpha(z)&=&\alpha(x)\vdash(y\dashv z) \label{eq6}\\
(x\dashv y)\vdash\alpha(z)&=&\alpha(x)\vdash(y\vdash z) \label{eq7}\\
(x\vdash y)\vdash\alpha(z)&=&\alpha(x)\vdash(y\vdash z).\label{eq8}
\end{eqnarray}
\end{definition}
\begin{definition}\label{d3}
A Hom-associative algebra is a triple $(D ,\  \mu,\   \alpha)$  where
$D$ is a $\mathbb{K}$-vector space  \\$\mu :\ D \times D \to D ,\  \mu(x,\  y) = x \cdot y$ is a bilinear map   and $\alpha : \ D \to D $ is a linear map  such that the following Hom-associativity identity holds for all $x,\ y,\  z \in D$ we have
$\alpha(x) \cdot (y \cdot z) = (x \cdot y) \cdot \alpha(z)$. \\$(D ,\  \mu,\   \alpha)$ is said to be a multiplicative Hom-algebra if $\forall x ,\ y \in D$ we have $\alpha(x\cdot y)=\alpha(x)\cdot\alpha(y).$\\
Hom-diassociative algebras arise naturally from averaging operators on Hom-associative algebras:\\
 Let $(A ,\ \mu ,\ \alpha)$ be an Hom-associative algebra and $H$ be an averaging operator of it.\\ In other words 
 $H :\ A\to A$ is a linear map satisfying
 $\mu(Ha ,\ Hb) = H\mu(a ,\ Hb) = H\mu(Ha,\  b)  ,\ \forall\ a  ,\b \in A.$\\
We generalize to Hom-algebras setting. the result given by Aguiar in \cite{AG} for diassociative algebra. \\If  \( H \) is an averaging operator on \( (A ,\ \mu ,\ \alpha) \)  then the two operations \( \dashv \) and \( \vdash \) defined on \( A \)  $$a \dashv b =\mu(a,\  Hb)  \quad a \vdash b = \mu(Ha,\  b) $$
define a Hom-diassociative algebra structure on \( A \).
\end{definition}
\begin{definition}\label{d7}
Let $(D,\  \dashv ,\ \vdash,\  \alpha)$ be a Hom-diassociative algebra. A linear operator $R :\ D \to D $ is called a Rota-Baxter operator of weight zero if 
\begin{eqnarray}
 R \circ \alpha&=&\alpha \circ R\\
 R(x)\dashv R(y)&=&R(R(x) \dashv y + x \dashv R(y))\\
R(x)\vdash R(y) &=&R(R(x) \vdash y + x \vdash R(y))
\end{eqnarray}
  for all $x,\ y \in D$.
\end{definition}
\begin{proposition}
Let $(D  ,\ \dashv,\   \vdash ,\  \alpha)$ be a Hom-diassociative algebra and $R :\ D\longrightarrow D$ be a Rota-Baxter operator of weight $0$ on D i.e R is linear and commutes with both $\alpha$
and
\begin{eqnarray*}
 R(x)\dashv R(y)&=&R(R(x) \dashv y + x \dashv R(y))\\
R(x)\vdash R(y) &=&R(R(x) \vdash y + x \vdash R(y)).
\end{eqnarray*} 
Then  $(D  ,\ \unlhd,\  \unrhd,\  \alpha)$ is also Hom-diassociative algebra with:\\$x\unlhd y=R(x) \dashv y + x \dashv R(y)$ and $x\unrhd y =R(x) \vdash y + x \vdash R(y)  ,\  \forall x,\  y\in D.$ \\ Moreover  $R$ is a Rota-Baxter operator on $(D ,\ \unlhd,\  \unrhd,\  \alpha).$ 
\end{proposition}
\begin{proof}
We prove axiom (\ref{eq6})  the other being proved in a similar way. Thus  for any $x,\  y  ,\ z\in D$
\begin{align*}
(x \unlhd y)\unlhd\alpha(z)-\alpha(x)\unlhd(y \unrhd z)&=R(x\dashv R(y)+ R(x)\dashv y)\dashv\alpha(z)+R(R(x)\dashv y + x\dashv R(y)\dashv \alpha(z)\\
&-\alpha(x)\dashv R(R(y)\vdash z+y\vdash R(z))+R(\alpha(x))\dashv(R(y)\vdash z +y\vdash R(z))\\
&=(x \dashv R(y))\dashv R(\alpha (z))+(R(x)\dashv y)\dashv R(\alpha (z))+(R(x)\dashv R(y))\dashv\alpha(z)\\
&-\alpha(x)\dashv(R(y)\vdash R(z))-R(\alpha (x))\dashv(y\vdash R(z))-R(\alpha (x))\dashv(R(y)\vdash z)\\
&=0
\end{align*}
The left hand side vanishes by axiom (\ref{eq6}). The ends the proof.
\end{proof}
\begin{ex}
Let $(D ,\ \dashv,\  \vdash,\  \alpha)$ be a Hom-diassociative algebra with a basis $\left\{e_1 ,\ e_2 ,\ e_3\right\}$ whose non-zero product are given by 
$$
D= \left\{
    \begin{array}{ll}
      e_1\dashv e_2=-2e_3&e_1\vdash e_1=\frac{1}{4}e_3\\
      e_2\dashv e_1=e_3&e_1\vdash e_2=e_3\\
			e_2\dashv e_2=e_3&e_2\vdash e_2=e_3 
    \end{array}
\right.\quad
\alpha = 
\begin{pmatrix}
a & 1 &0 \\
0 & a &0 \\
0 &0&1 
\end{pmatrix} ,\ a\in\mathbb{R}.
$$
We have the operator Rota-Baxter:
$$
\begin{pmatrix}
0 &0 &0 \\
0 & 0 &0 \\
0 & r_{32}& r_{33}
\end{pmatrix} \quad
\begin{pmatrix}
0 &0 &0 \\
0 & 0 &0 \\
0 & r_{32}&0
\end{pmatrix} \quad
\begin{pmatrix}
2r_{33} &0 &0 \\
0 & 2r_{33} &0 \\
0 & r_{32}& r_{33}
\end{pmatrix}  ,\ r_{ij}\in\mathbb{R}.
$$
\end{ex}
\section{Hom-quadri-dendriform algebra}
In this section  we introduce the notions of Hom-quadri-dendriform algebras generalizing the classical quadri-dendriform algebras to Hom-algebras setting. Next  we introduce relative averaging
 operators on Hom-dendriform algebras. We show that a relative averaging operator induces a
 Hom-quadri-dendriform algebra. Finally  we show that any Hom-quadri-dendriform algebra can be
 embedded into an averaging Hom-dendriform algebra.
\begin{definition}\label{d8} 
A Hom-quadri-dendriform algebra is a tuple 
$(D,\  \prec_\vdash,\  \prec_\dashv,\  \succ_\vdash,\  \succ_\dashv,\  \alpha) $  where $D$ is a vector space  the operations $\prec_\vdash  ,\ \prec_\dashv ,\  \succ_\vdash,\   \ \succ_\dashv: \ D \times D \to D$ are bilinear  and $\alpha: \ D \to D$ is a linear map  
such that for all $x,\   y  ,\ z \in D$  the following identities hold
\begin{align}
(x \prec_\vdash y) \prec_\vdash \alpha(z) &= (x \prec_\dashv y) \prec_\vdash \alpha(z) = \alpha(x) \prec_\vdash (y \prec_\vdash z + y \succ_\vdash z) \label{Hq1} \\
(x \succ_\vdash y) \prec_\vdash \alpha(z) &= (x \succ_\dashv y) \prec_\vdash \alpha(z) = \alpha(x) \succ_\vdash (y \prec_\vdash z)\label{Hq2} \\
\alpha(x) \succ_\vdash (y \succ_\vdash z)&=(x \prec_\vdash y + x \succ_\vdash y)\succ_\vdash \alpha(z)=(x \prec_\dashv y + x \succ_\dashv y) \succ_\vdash \alpha(z) \label{Hq3} \\
\alpha(x) \succ_\vdash (y \succ_\vdash z)&=(x \prec_\dashv y + x \succ_\vdash y) \succ_\vdash \alpha(z)=(x \prec_\vdash y + x \succ_\dashv y) \succ_\vdash \alpha(z) \label{Hq4} \\
(x \prec_\vdash y) \prec_\dashv \alpha(z)& = \alpha(x) \prec_\vdash (y \prec_\dashv z + y \succ_\dashv z) \label{Hq5} \\
(x \succ_\vdash y) \prec_\dashv \alpha(z)&= \alpha(x) \succ_\vdash (y \prec_\dashv z) \label{Hq6} \\
\alpha(x) \succ_\vdash (y \succ_\dashv z)& = (x \prec_\vdash y + x \succ_\vdash y) \succ_\dashv \alpha(z) \label{Hq7} \\
(x \prec_\dashv y) \prec_\dashv \alpha(z)&=\alpha(x) \prec_\dashv (y \prec_\vdash z + y \succ_\vdash z)=\alpha(x) \prec_\dashv (y \prec_\dashv z + y \succ_\dashv z)\label{Hq8} \\
(x \prec_\dashv y) \prec_\dashv \alpha(z)&=\alpha(x) \prec_\dashv (y \prec_\vdash z + y \succ_\dashv z)=\alpha(x) \prec_\dashv (y \prec_\dashv z + y \succ_\vdash z) \label{Hq9} \\
(x \succ_\dashv y) \prec_\dashv \alpha(z) &= \alpha(x) \succ_\dashv (y \prec_\vdash z) = \alpha(x) \succ_\dashv (y \prec_\dashv z)\label{Hq10} \\
\alpha(x)\succ_\dashv (y \succ_\vdash z) &= \alpha(x) \succ_\dashv (y \succ_\dashv z) = (x \prec_\dashv y + x \succ_\dashv y) \succ_\dashv \alpha(z)  \label{Hq11}\\
\end{align}
$(D,\  \prec_\vdash,\  \prec_\dashv,\  \succ_\vdash,\  \succ_\dashv,\  \alpha) $ is said to be a multiplicative Hom-algebra if $\forall x,\  y,\ \in D $ we have 
\begin{align*}
    \alpha(x \prec_\vdash y) &= \alpha(x) \prec_\vdash' \alpha(y)  ,\quad
\alpha(x \prec_\dashv y) = \alpha(x) \prec_\dashv' \alpha(y) ,\\
\alpha(x \succ_\vdash y) &= \alpha(x) \succ_\vdash' \alpha(y)  ,\quad 
\alpha(x \succ_\dashv y) = \alpha(x) \succ_\dashv' \alpha(y).
\end{align*}
\end{definition}
\begin{definition}
Given two Hom-quadri-dendriform algebras$(D,\  \prec_\vdash,\  \prec_\dashv,\  \succ_\vdash,\  \succ_\dashv,\  \alpha) $and $(D',\  \prec_\vdash',\  \prec_\dashv',\  \succ_\vdash',\  \succ_\dashv',\  \alpha') $. A Hom-quadri-dendriform homomorphism is a linear map $T : \ D \to D'$ that satisfies the following compatibility conditions
\begin{align*}
T(x \prec_\vdash y) &=T(x) \prec_\vdash' T(y) ,\quad T(x \prec_\dashv y) = T(x) \prec_\dashv'T(y) ,\\
T(x \succ_\vdash y) &=T(x) \succ_\vdash' T(y) ,\quad T(x \succ_\dashv y) =T(x) \succ_\dashv' T(y),\\ 
T \circ \alpha&=\alpha' \circ T.
\end{align*}
for \ all\ $x ,\  y\in D$. In particular  if $T$ is bijective  then \\ $(D,\  \prec_\vdash,\  \prec_\dashv,\  \succ_\vdash,\  \succ_\dashv,\  \alpha) $ and $(D',\  \prec_\vdash',\  \prec_\dashv',\  \succ_\vdash',\  \succ_\dashv',\  \alpha') $are said to be isomorphic.
\end{definition}
\begin{proposition}\label{qpro} 
Let $(D, \prec_\vdash,\prec_\dashv,\succ_\vdash,\succ_\dashv,\alpha) $ be an a Hom-quadri-dendriform algebra. Then $(D,\vdash,\dashv,\alpha)$ is a Hom-diassociative algebra  where
\begin{align*}
x \vdash y &= x \prec_\vdash y + x \succ_\vdash y  \\
x \dashv y &= x \prec_\dashv y + x \succ_\dashv y  
\end{align*}
for all $x ,\ y \in D$.
\end{proposition}
\begin{proof}
$\forall x,\  y,\  z \in D $ we have 
\begin{align*}
(x \dashv y) \dashv \alpha(z) 
&= (x \prec_\dashv y + x \succ_\dashv y) \succ_\dashv \alpha(z) + (x \prec_\dashv y + x \succ_\dashv y) \prec_\dashv \alpha(z) \\
&= (x \prec_\dashv y + x \succ_\dashv y) \succ_\dashv \alpha(z) + (x \prec_\dashv y) \prec_\dashv \alpha(z) + (x \succ_\dashv y) \prec_\dashv \alpha(z) \\
&= (x \prec_\dashv y + x \succ_\dashv y) \succ_\dashv \alpha(z) + \alpha(x) \prec_\dashv (y \prec_\dashv z + y \succ_\dashv z) + (x \succ_\dashv y) \prec_\dashv \alpha(z) \\
&= \alpha(x) \succ_\dashv (y \prec_\dashv z) + \alpha(x) \succ_\dashv (y \succ_\dashv z) + \alpha(x) \prec_\dashv (y \prec_\dashv z + y \succ_\dashv z) \\
&= \alpha(x) \dashv (y \dashv z), \\
\\
(x \dashv y) \dashv \alpha(z) 
&= (x \prec_\dashv y + x \succ_\dashv y) \succ_\dashv \alpha(z) + (x \prec_\dashv y) \prec_\dashv \alpha(z) \\
&= (x \prec_\dashv y + x \succ_\dashv y) \succ_\dashv \alpha(z) + \alpha(x) \prec_\dashv (y \prec_\vdash z + y \succ_\vdash z) + (x \succ_\dashv y) \prec_\dashv \alpha(z) \\
&= \alpha(x) \prec_\dashv (y \prec_\vdash z + y \succ_\vdash z) + \alpha(x) \succ_\dashv (y \prec_\vdash z) + \alpha(x) \succ_\dashv (y \succ_\vdash z) \\
&= \alpha(x) \dashv (y \vdash z), \\
\\
(x \vdash y) \dashv \alpha(z) 
&= (x \prec_\vdash y + x \succ_\vdash y) \prec_\dashv \alpha(z) + (x \prec_\vdash y + x \succ_\vdash y) \succ_\dashv \alpha(z) \\
&= \alpha(x) \succ_\vdash (y \succ_\dashv z) + (x \prec_\vdash y) \prec_\dashv \alpha(z) + (x \succ_\vdash y) \prec_\dashv \alpha(z) \\
&= \alpha(x) \succ_\vdash (y \succ_\dashv z) + \alpha(x) \succ_\vdash (y \prec_\dashv z) + \alpha(x) \prec_\vdash (y \succ_\dashv z + y \prec_\dashv z) \\
&= \alpha(x) \vdash (y \dashv z), \\
\\
(x \dashv y) \vdash \alpha(z) 
&= (x \prec_\dashv y + x \succ_\dashv y) \prec_\vdash \alpha(z) + (x \prec_\dashv y + x \succ_\dashv y) \succ_\vdash \alpha(z) \\
&= (x \prec_\dashv y + x \succ_\dashv y) \prec_\vdash \alpha(z) + \alpha(x) \succ_\vdash (y \succ_\vdash z) \\
&= \alpha(x) \succ_\vdash (y \succ_\vdash z) + \alpha(x) \prec_\vdash (y \prec_\vdash z + y \succ_\vdash z) + \alpha(x) \succ_\vdash (y \succ_\vdash z) \\
&= \alpha(x) \vdash (y \vdash z), \\
\\
(x \vdash y) \vdash \alpha(z) 
&= (x \prec_\vdash y + x \succ_\vdash y) \prec_\vdash \alpha(z) + (x \prec_\vdash y + x \succ_\vdash y) \succ_\vdash \alpha(z) \\
&= (x \prec_\vdash y) \prec_\vdash \alpha(z) + (x \succ_\vdash y) \prec_\vdash \alpha(z) + \alpha(x) \succ_\vdash (y \succ_\vdash z) \\
&= \alpha(x) \prec_\vdash (y \prec_\vdash z + y \succ_\vdash z) + \alpha(x) \succ_\vdash (y \prec_\vdash z) + \alpha(x) \succ_\vdash (y \succ_\vdash z) \\
&= \alpha(x) \vdash (y \vdash z).
\end{align*}

The end of the proof.
\end{proof}
\begin{definition}\label{d9}
Let $(D,\  \prec_\vdash,\  \prec_\dashv,\  \succ_\vdash,\  \succ_\dashv ,\ \alpha)$ be a Hom-quadri-dendriform algebra and $D_0$ subset of $D$\\
We say that $D_0$ is a Hom-quadri-dendriform subalgebra of $D$ if $D_0$ is  stable under $\alpha, \  \ \prec_\vdash,\  \ \prec_\dashv,\   \succ_\vdash$ and $\succ_\dashv$, i.e. $\forall\ x,\  y  ,\ \ z \ \in \ D_0 ,\ \alpha(x) \in D_o,\  x\prec_\vdash y\in D_o , \ x\prec_\dashv y\in D_o  ,\ x\succ_\vdash y\in D_o$ and $x\succ_\dashv y \in D_0.$
\end{definition}
\begin{ex}\label{exemple3.5}
  Let $(D,\  \prec_\vdash,\  \prec_\dashv,\  \succ_\vdash,\  \succ_\dashv,\  \alpha) $  be a Hom-quadri-dendriform algebra. The vector subspace $I_D$ spanned by
  \[
\left\{
x \prec_\dashv y - x \prec_\vdash y ,\quad 
x \succ_\dashv y - x \succ_\vdash y 
\;\middle|\; x,\ y \in D
\right\}
\]
 is an ideal of $D$. It is clear that $D$ is a Hom-dendriform algebra if and only if  $I_D=\left\{0\right\}$. Therefore  the quotient algebra $(D / I_D ,\  \prec ,\ \succ ,\ \overline{\alpha})$ is a Hom-dendriform algebra where\\ 
 $\overline{\alpha}(\overline{x})=\overline{\alpha(x)} ,\quad
\overline{x} \prec \overline{y} = \overline{x \prec_{\vdash} y} = \overline{x \prec_{\dashv} y}  ,\quad
\overline{x} \succ \overline{y} = \overline{x \succ_{\vdash} y}= \overline{x \succ_{\dashv} y}$,\  
 for all $x,\  y \in D$ \ where $\overline{x}$ denotes the class of an element $x \in D$.
\end{ex}
\begin{proposition}
Let $(A ,\ \prec_{\vdash_{A}},\  \prec_{\dashv_{A}} ,\ \succ_{\vdash_{A}} ,\ \succ_{\dashv_{A}},\  \alpha_{A})$ and $(B  ,\ \prec_{\vdash_{B}},\  \prec_{\dashv_{B}},\  \succ_{\vdash_{B}} ,\ \succ_{\dashv_{B}} ,\ \alpha_{B})$ be two Hom-quadri-dendriform algebras. Then  there exists a Hom-quadri-dendriform algebra structure on $A\oplus B$ with the bilinear maps:
$\lhd , \  \unlhd,\  \rhd,\  \unrhd : \ (A\oplus B)^{\otimes2}\longrightarrow A\oplus B$ given by  
\begin{align*}
 (a_1+b_1)\prec_\vdash(a_2+b_2):&=a_1\prec_{\vdash_{A}}a_2+b_2\prec_{\vdash_{B}}b_2 \\
  (a_1+b_1)\prec_\dashv(a_2+b_2):&=a_1\prec_{\dashv_{A}}a_2+b_2\prec_{\dashv_{B}}b_2 \\
	(a_1+b_1)\succ_\vdash(a_2+b_2):&=a_1\succ_{\vdash_{A}}a_2+b_2\succ_{\vdash_{B}}b_2 \\
  (a_1+b_1)\succ_\dashv(a_2+b_2):&=a_1\succ_{\dashv_{A}}a_2+b_2\succ_{\dashv_{B}}b_2 
\end{align*}
and the linear map
 $\alpha=\alpha_A+\alpha_B :\ A\oplus B\longrightarrow A\oplus B$ given by \\
$$(\alpha_A+\alpha_B)(a+b):=\alpha_A(a)+\alpha_B(b)   ,\ \forall(a,\ b)\in A\times B.$$

Moreover  if $\xi : \ A\longrightarrow B$ is a linear map. Then  
\begin{align*}
 \xi  : \ (A ,\ \prec_{\vdash_{A}},\  \prec_{\dashv_{A}},\  \succ_{\vdash_{A}},\  \succ_{\dashv_{A}} ,\ \alpha_{A})\longrightarrow(B  ,\ \prec_{\vdash_{B}} ,\ \prec_{\dashv_{B}},\  \succ_{\vdash_{B}},\  \succ_{\dashv_{B}} ,\ \alpha_{B})
\end{align*}
is a morphism if and only if its graph $\Gamma_\xi=\left\{(x  ,\ \Gamma(x) ,\ x\in A\right\}$ is a Hom-quadri-dendriform algebra of 
$(A\oplus B,\  \prec_\vdash ,\  \prec_\dashv  ,\ \succ_\vdash,  \ \succ_\dashv ,\ \alpha).$
\end{proposition}
\begin{proof}
The proof of the first part comes from a direct computation  so we omit it. Now  let us suppose that 
$\xi :\  (A , \ \prec_{\vdash_{A}} ,\ \prec_{\dashv_{A}},\   \succ_{\vdash_{A}} ,\ \succ_{\dashv_{A}},\  \alpha_{A})\longrightarrow(B ,\ \prec_{\vdash_{B}} ,\ \prec_{\dashv_{B}} ,\ \succ_{\vdash_{B}},\  \succ_{\dashv_{B}} ,\ \alpha_{B})$ is a morphism of Hom-quadri-dendriform algebras. Then 
\begin{align*}
 (x+\xi(x))\prec_\vdash(y+\xi(y)):&=(x\prec_{\vdash_{A}}y+\xi(x)\prec_{\vdash_{B}}\xi(y))=x\prec_{\vdash_{A}} y+\xi(x\prec_{\vdash_{A}} y) \\
 (x+\xi(x))\prec_\dashv(y+\xi(y)):&=(x\prec_{\dashv_{A}}y+\xi(x)\prec_{\dashv_{B}}\xi(y))=x\prec_{\dashv_{A}} y+\xi(x\prec_{\dashv_{A}} y) \\
 (x+\xi(x))\succ_\vdash(y+\xi(y)):&=(x\succ_{\vdash_{A}}y+\xi(x)\succ_{\vdash_{B}}\xi(y))=x\succ_{\vdash_{A}} y+\xi(x\succ_{\vdash_{A}} y) \\
 (x+\xi(x))\succ_\dashv(y+\xi(y)):&=(x\succ_{\dashv_{A}}y+\xi(x)\succ_{\dashv_{B}}\xi(y))=x\succ_{\dashv_{A}} y+\xi(x\succ_{\dashv_{A}} y).
\end{align*}
Thus the graph $\Gamma_{\xi}$ is closed under the operations $\prec_\vdash,\  \prec_\dashv,\  \succ_\vdash$ and $\succ_\dashv.$\\
Furthermore since $\xi\circ\alpha_A=\alpha_B\circ\xi $ we have 
\begin{align*}
 (\alpha_A\oplus\alpha_B)(x ,\ \xi(x))=(\alpha_A(x),\ \alpha_B\circ\xi(x))=(\alpha_A(x),\
  \xi\circ\alpha_A(x)).
\end{align*}
Which implie that $\Gamma_\xi$ is closed $\alpha_A\oplus\alpha_B.$ Thus  $\Gamma_\xi$ is a Hom-subalagebra of $(A\oplus B ,\ \prec_\vdash ,\ \prec_\dashv,\  \succ_\vdash ,\ \succ_\dashv ,\ \alpha).$ 
Conversely  if the graph $\Gamma_\xi\subset A\oplus B$ is a Hom-subalgebra of $(A\oplus B ,\ \prec_\vdash ,\ \prec_\dashv,\  \succ_\vdash,\  \succ_\dashv,\  \alpha)$ then we have
\begin{align*}
 (x+\xi(x))\prec_\vdash(y+\xi(y)):&=(x\prec_{\vdash_{A}}y+\xi(x)\prec_{\vdash_{B}}\xi(y))\in\Gamma_\xi \\
 (x+\xi(x))\prec_\dashv(y+\xi(y)):&=(x\prec_{\dashv_{A}}y+\xi(x)\prec_{\dashv_{B}}\xi(y))\in\Gamma_\xi \\
 (x+\xi(x))\succ_\vdash(y+\xi(y)):&=(x\succ_{\vdash_{A}}y+\xi(x)\succ_{\vdash_{B}}\xi(y))\in\Gamma_\xi \\
 (x+\xi(x))\succ_\dashv(y+\xi(y)):&=(x\succ_{\dashv_{A}}y+\xi(x)\succ_{\dashv_{B}}\xi(y))\in\Gamma_\xi.
\end{align*}
Furthermore  $(\alpha_A\oplus\alpha_B)(\Gamma_\xi)\subset\Gamma_\xi $ implies $$ (\alpha_A\oplus\alpha_B)(x,\ \xi(x))=(\alpha_A(x) ,\ \alpha_B\circ\xi(x))\in\Gamma_\xi$$

which is equivalent to the condition $\alpha_B\circ\xi(x)=\xi\circ\alpha_A(x) $.i.e \ $\alpha_B\circ\xi=\xi\circ\alpha_A.$ Therefore  $\xi$ is a morphism of Hom-quadri-dendriform algebras.
\end{proof}

%%%%%%%%%%%%%%%%%%%%%%%%%%%%%%%%%%%%%%%%%%%%%%%%%%%%%%%%%%%%%%%%%%%%
%%%%%%%%%%%%%%%%%%%%%%%%%%%%%%%%%%%%%%%%%%%%%%%%%%%%%%%%%%%%%%%%%%%%
\begin{proposition}\label{rpro}
Let $(V ,\  \prec_l ,\  \succ_l,\   \prec_r ,\ \succ_r  ,\ \beta)$ be a representation of a Hom-dendriform algebra \\
$(D  ,\ \prec  ,\ \succ  ,\ \alpha)$. Then 
$(D \oplus V,\ \prec_\vdash,\  \prec_\dashv,\  \succ_\vdash,\  \succ_\dashv,\  \alpha \otimes \beta)$
is a Hom-quadri-dendriform algebra where 
\begin{align*}
(x ,\ u) \prec_\vdash (y ,\  v) &:= (x \prec y ,\ x \prec_l v)  \\
(x ,\ u) \prec_\dashv (y ,\  v) &:= (x \prec y ,\ u \prec_r y)  \\
(x ,\ u) \succ_\vdash (y ,\  v) &:= (x \succ y ,\ x \succ_l v)  \\
(x ,\ u) \succ_\dashv (y ,\  v) &:= (x \succ y ,\ u \succ_r y) .
\end{align*}
$\forall x,\ y \in D \ and\ u ,\ v \in V$. This Hom-quadri-dendriform algebra is called the hemi-semidirect product Hom-quadri-dendriform algebra.
\end{proposition}
\begin{proof}
For any $(x,\  u),\  (y,\  v)   \ (z ,\ w) \in D \oplus V$ and $\ast \in \{ \vdash ,\ 
 \dashv\} $  we have
\begin{align*}
((x,\ u)\prec_\ast(y,\ v)) \prec_\vdash (\alpha(z), \beta(w))
&= ((x \prec y) \prec \alpha(z),\ (x \prec y) \prec_\ell \beta(w)) \\
&= (\alpha(x) \prec (y \prec z + y \succ z),\ \alpha(x) \prec_\ell (v \prec_\ell w + v \succ_\ell w)) \\
&= (\alpha(x),\ \beta(u)) \prec_{\vdash} \big((y,\ v) \prec_{\vdash} (z,\ w) + (y,\ v) \succ_{\vdash} (z,\ w)), \\
\\
((x,\ u) \succ_{\ast} (y,\ v)) \prec_{\vdash} (\alpha(z), \beta(w))
&= ((x \succ y) \prec \alpha(z),\ (x \succ y) \prec_\ell \beta(w)) \\
&= (\alpha(x) \succ (y \prec z),\ \alpha(x) \succ_\ell (y \prec_\ell w)) \\
&= (\alpha(x),\ \beta(u)) \succ_{\vdash} ((y,\ v) \prec_{\vdash} (z,\ w)), \\
\\
(\alpha(x),\ \beta(u)) \succ_{\vdash} ((y,\ v) \succ_{\vdash} (z,\ w))
&= (\alpha(x) \succ (y \succ z),\ \alpha(x) \succ_\ell (y \succ_\ell w)) \\
&= ((x \prec y + x \succ y) \succ \alpha(z),\ (x \prec y + x \succ y) \succ_\ell \beta(w)) \\
&= ((x,\ u) \prec_{\ast} (y,\ v) + (x,\ u) \succ_{\ast} (y,\ v)) \succ_{\vdash} (\alpha(z), \beta(w)), \\
\\
((x,\ u)\prec_{\vdash}(y,\ v))\prec_{\dashv}(\alpha(z), \beta(w))
&= ((x \prec y) \prec \alpha(z),\ (x \prec_\ell v) \prec_r \alpha(z)) \\
&= (\alpha(x) \prec (y \prec z + y \succ z),\ \alpha(x) \prec_\ell (v \prec_r z + v \succ_r z)) \\
&= (\alpha(x), \beta(u)) \prec_{\vdash} \left((y,\ v) \prec_{\dashv}(z,\ w) + (y,\ v) \succ_{\dashv}(z,\ w)\right).
\end{align*}
Thus the Eq (\ref{Hq5}) hold.
 The same for Eqs (\ref{Hq1}) à (\ref{Hq11}) . This completes the proof.
\end{proof}
\begin{definition}\label{d10}
Let $(D ,\ \prec,\  \succ,\  \alpha)$ be a Hom-dendriform algebra and $(V,\   ,\ \prec_{\ell} ,\ \succ_{\ell} ,\  \prec_{r} ,\ \succ_{r}  ,\ \beta)$ be a representation of it. A relative averaging operator on $D$ with respect to the representation 
$(V,\   \prec_{\ell} ,\ \succ_{\ell} ,\ \prec_{r} ,\ \succ_{r},\ \beta)$ is a linear map $T :\ V \to D$ that satisfies
\begin{align*}
Tu \prec Tv& = T(Tu \prec_{\ell} v) = T(u \prec_{r} Tv), \\
Tu \succ Tv &= T(Tu \succ_{\ell} v) = T(u \succ_{r} Tv), \\
T \circ \beta &= \alpha \circ T 
\end{align*}
forall $u ,\ v \in V$. The relative averaging operator on a Hom-dendriform algebra  $(D ,\ \prec,\  \succ ,\ \alpha)$  with respect to the
 adjoint representation $(D  ,\ \prec_{\ell}=\prec ,\ \ \succ_{\ell}=\succ ,\ \prec_{r}=\prec  ,\ \succ_{r}=\succ,\alpha)$  is called an averaging operator. In this case the Eqs can be written as 
\begin{align*}
Tx \prec Ty& = T(Tx \prec y) = T(x \prec Ty) \\
Tx \succ Ty&= T(Tx \succ y) = T(x \succ Ty)\\
T\circ \alpha&=\alpha\circ T.
\end{align*}
 \end{definition}
\begin{definition}\label{d11}
An averaging Hom-dendriform  algebra  is a  pair $(D ,\ T)$  consisting  of  a  Hom-dendriform  algebra $(D ,\ \prec,\  \succ ,\ \alpha)$ endowed with an averaging operator $T$.
\end{definition}
In the following  we give a characterization of a relative averaging operator in terms
of the graph of the operator.
\begin{thm}
A linear\ map  $T :\ V \to D$ is a relative averaging operator on a Hom-dendriform algebra $(D ,\ \prec,\  \succ ,\ \alpha)$ with respect to the representation 
$(V ,\ \prec_{\ell} ,\ \succ_{\ell},\  \prec_{r} ,\ \succ_{r},\  \beta)$ if and only if the graph\\
$Gr(T) = \left\{(Tu,\  u) \mid u \in V\right\}$
is a Hom-subalgebra of the hemi-semidirect product Hom-quadri-dendriform \\ algebra $(D \oplus V ,\ \prec_\vdash,\   \prec_\dashv,\   \succ_\vdash,\   \succ_\dashv,\  \alpha \otimes \beta).$
\end{thm}
\begin{proof}
Let $ T : \ V \to D $ be a linear map. For any $ u  ,\ v \in V$  we have
\begin{align*}
(Tu ,\ u) \prec_\vdash (Tv ,\ v) =& (Tu \prec Tv , Tu \prec_l v),\quad
(Tu ,\ u) \prec_\dashv (Tv ,\ v) = (Tu \prec Tv  u \prec_r Tv) ,\\
(Tu ,\ u) \succ_\vdash (Tv ,\ v) = &(Tu \succ Tv,\  Tu \succ_l v),\quad (Tu ,\ u) \succ_\dashv (Tv ,\ v) = (Tu \succ Tv ,\ u \succ_r Tv).
\end{align*}
The  above  five  elements  are  in  $\operatorname{Gr}(T)$  if   and  only   if \\ 
$Tu \prec Tv = T(Tu \prec_l v)  ,\quad Tu \prec Tv = T(u \prec_r Tv),\quad
Tu \succ Tv = T(Tu \succ_l v) ,\\
Tu \succ Tv = T(u \succ_r Tv),\quad 
\alpha \otimes \beta((Tu  ,\ u))=(\alpha(Tu),\  \beta(u))=(T(\beta(u)) ,\ \beta(u)) \in \operatorname{Gr}(T)$.\\
Therefore  the graph ${Gr}(T)$  is a Hom-subalgebra of the hemi semidirect product Hom-quadri-dendriform algebra $(D \oplus V,\  \prec_\vdash ,\ \prec_\dashv,\  \succ_\vdash  ,\ \succ_\dashv ,\ \alpha \otimes \beta)$  if   and   only  if $ T $ is a relative averaging operator.
\end{proof}

\begin{proposition}\label{prop3.11}
Let  $T : \ V \to D$ be a relative averaging operator on a Hom-dendriform algebra \\$(D ,\ \prec,\  \succ ,\ \alpha)$ with respect to the representation $(V ,\  \prec_{\ell} ,\ \succ_{\ell},\  \prec_{r}  ,\ \succ_{r} ,\ \beta)$. Then $V$ inherits a Hom-quadri-dendriform algebra structure with the operations:
\begin{align*}
u \prec^{T}_{\vdash} v &= T(u) \prec_{\ell} v ,\quad u\prec^{T}_{\dashv} v = u \prec_{r} T(v) \\
u\succ^{T}_{\vdash} v &= T(u) \succ_{\ell} v ,\quad u \succ^{T}_{\dashv} v = u \succ_{r} T(v) 
\end{align*}
for all $u ,\ v \in V$.  Moreover  $T$  is  a  homomorphism from  the Hom-quadri-dendriform  algebra \\$(V ,\ \prec^{T}_{\vdash} ,\ \prec^{T}_{\dashv},\  \succ^{T}_{\vdash},\  \succ^{T}_{\dashv},\  \beta)$ to  the  Hom-dendriform algebra $(D ,\  \prec  ,\ \succ ,\ \alpha).$
\end{proposition}
\begin{cor}
Let $(D ,\ \prec,\  \succ ,\ \alpha)$ be a Hom-dendriform algebra and let $ T :\ D \to D $ be an  averaging operator. Then  $D$ is also equipped with a Hom-quadri-dendriform algebra structure defined by\\
$x \prec^{T}_\vdash y = T(x) \prec y ,\quad x \prec^{T}_\dashv y = x \prec T(y) ,\quad x \succ^{T}_\vdash y = T(x) \succ y  ,\quad x \succ^{T}_\dashv y = x \succ T(y)$ ,\
 $\forall x,\ y \in D.$
\end{cor}
\begin{thm}
 Every Hom-quadri-dendriform algebra can  induce a relative  averaging  operator on  a\\  Hom-dendriform algebra with respect to a representation.
\end{thm}
\begin{proof}
  Let $(D ,\ \prec_\vdash,\  \prec_\dashv,\   \succ_\vdash,\   \succ_\dashv,\  \alpha)$ be an a Hom-quadri-dendriform algebra  Then by Example \ref{exemple3.5}  \\$D_{\text{Dend}} = \frac{D}{\mathbb{I}_D}$  
carries a Hom-dendriform algebra structure with the operations
\begin{align*}
\overline{\alpha(\overline{x})}&=\overline{\alpha(x)},\\
\overline{x} \prec \overline{y} &=\overline{x \prec_{\vdash} y} = \overline{x \prec_{\dashv} y}  ,\\
\overline{x} \succ \overline{y} &= \overline{x \succ_{\vdash} y}= \overline{x \succ_{\dashv} y}.
\end{align*}
 for all $x,\ y \in D$  here $\overline{x}$ denotes the class of an element $x \in D$. We define four bilinear
 maps\\ $\prec_l ,\  \succ_l :\ D_{\text{Dend}} \times D \to D$
and $\prec_r ,\ \succ_r :\ D \times D_{\text{Dend}} \to D$ by\\ \[
x \prec_l y = x \prec_{\vdash} y  ,\quad
x \succ_l y = x \succ_{\vdash} y  ,\quad
y \prec_r x = y \prec_{\dashv} x  ,\quad
y \succ_r x = y \succ_{\dashv} x.
\]
for all $\overline{x} \in D_{\text{Dend}}$ and $y \in D$. It is easy to verify that the maps $\prec_l,\  \succ_l ,\ \prec_r,\  \succ_r$ define a representation of the Hom-dendriform algebra $D_{\text{Dend}}$ on the vector space $D$. Moreover  the quotient map $T :\ D \to D_{\text{Dend}}  \quad x \mapsto \overline{x}$ is a relative averaging operator as\\
\[
T x \prec T y =\overline{x} \prec \overline{y}=
 \left\{
\begin{array}{l}= \overline{x \prec_{\vdash} y}
= \overline{\overline{x} \prec_l y} = T(T x \prec_l y)  \\
= \overline{x \prec_{\dashv} y} 
= \overline{x \prec_r \overline{y}} = T(x \prec_r T y). 
\end{array}
\right.
\]
\[
T x \succ T y =\overline{x} \succ \overline{y}=
\left\{
\begin{array}{l}= \overline{x \succ_{\vdash} y}
= \overline{\overline{x} \succ_l y} = T(T x \succ_l y) \\
= \overline{x \succ_{\dashv} y}
= \overline{x \succ_r \overline{y}} = T(x \succ_r T y) 
\end{array}
\right.
\]
for all $x,\ y \in D.$
\end{proof}

\section{Hom-Six-dendriform algebras}
In \cite{LR}  Loday and Ronco introduced the notion of dendriform trialgebra. Inspired by Loday and Ronco’s work  we present the notion  of Hom-six-dendriform algebras 
generalizing the classical six-dendriform algebras to Hom-algebras setting. 

 \begin{definition}\label{d12}
 Let $(D ,\ \prec,\  \succ ,\ \alpha)$ and $(D' ,\ \prec',\  \succ ',\ \alpha')$ be two Hom-dendriform algebras. We say that $D$ acts on $D'$ if there exist four bilinear maps\
$
\prec_l,\  \succ_l :\ D \otimes D' \to D' \quad \text{and} \quad \prec_r ,\ \succ_r :\ D' \otimes D \to D'$
such that $(D' ,\ \prec_l ,\ \succ_l ,\ \prec_r  ,\ \succ_r ,\ \alpha')$ is a representation of $D$ and for any $x,\ y,\  z \in D$ and $u  ,\ v,\  w \in D'$ we have
\begin{align}
(x \prec_l v) \prec' \alpha'(w)& = \alpha(x) \prec_l (v \prec' w + v \succ' w) \label{eqtt}\\
(x \succ_l v) \prec' \alpha'(w)& = \alpha(x) \succ_l (v \prec' w) \label{rq2} \\
\alpha(x) \succ_l (v \succ' w)& = (x \prec_l v + x \succ_l v) \succ' \alpha'(w) \label{rq3}\\
(u \prec_r y) \prec' \alpha'(w)& = \alpha'(u) \prec' (y \prec_l w + y \succ_l w) \label{rq4} \\
(u \succ_r y) \prec' \alpha'(w)&= \alpha'(u) \succ' (y \prec_l w) \label{rq5}\\
\alpha'(u)\succ' (y \succ_l w)& = (u \prec_r y + u \succ_r y) \succ' \alpha'(w) \label{rq6}\\
(u \prec' v) \prec_r \alpha(z)& =\alpha'(u) \prec'  (v \prec_r z + v \succ_r z) \label{rq7}\\
(u \succ' v) \prec_r \alpha(z)&= \alpha'(u) \succ' (v \prec_r z) \label{rq8}\\
\alpha'(u) \succ' (v \succ_r z)&= (u \prec' v + u \succ' v) \succ_r \alpha(z).\label{eqttc}
\end{align}
\end{definition}

\begin{proposition}
 Let $(D ,\ \prec,\  \succ ,\ \alpha)$ and $(D',\prec',\succ', \alpha')$ be two Hom-dendriform algebras and let \\$\prec_l,\succ_l: D \otimes D' \to D'$ and $\prec_r ,\ \succ_r :\ D' \otimes D \to D'$ be bilinear maps. Then  the tuple $(D', \prec_l, \succ_l, \prec_r,\succ_r, \alpha')$ is an action of $D$ if and only if $(D \oplus D', \alpha\otimes \alpha')$ carries a Hom-dendriform algebra structure with operations given by
 \[
(x,\  u) \mathbin{\prec_{\bowtie}} (y,\  v)  = \left(x \prec y,\  x \prec_l v + u \prec_r y + u \prec' v\right)  
\]
\[
(x,\  u) \mathbin{\succ_{\bowtie}} (y,\  v)  = \left(x \succ y,\  x \succ_l v + u \succ_r y + u \succ' v\right).
\]
for all $(x,\  u),\ (y ,\  v) \in (D \oplus D')$.\\ $(D \oplus D'  ,\ \alpha\otimes \alpha')$ is called the semi-direct product of $D$ with $D'$.\\
\end{proposition}
\begin{proof}
 Let $(D ,\ \prec,\  \succ ,\ \alpha)$ and $(D'  ,\ \prec'  ,\ \succ'  ,\ \alpha')$ be two Hom-dendriform algebras.\\ Then  for all $(x,\  u) ,\ (y  ,\ v),\  (z ,\ w) \in (D \oplus D') $ we have
 \begin{align*}
&((x,\  u) \prec_{\bowtie} (y,\  v) ) \prec_{\bowtie}(\alpha(z),\  \alpha'(w)) - (\alpha(x) ,\ \alpha'(u))\prec_{\bowtie} ((y,\  v)  \prec _{\bowtie}(z ,\ w) + (y,\  v)  \succ _{\bowtie}(z ,\ w))\\
&=( (x \prec y) \prec \alpha(z) - \alpha(x) \prec (y \prec z + y \succ z),\  (x \prec y) \prec_l \alpha'(w) 
+ (x \prec_l v + u \prec_r y + u \prec' v) \prec_r \alpha(z)\\
&+ (x \prec_l v + u \prec_r y + u \prec' v) \prec' \alpha'(w)- \alpha(x) \prec_l (y \prec_l w + y \succ_l w + v \prec_r z + v \succ_r z + v \prec' w + v \succ' w)\\
&- \alpha'(u) \prec_r (y \prec z + y \succ z)- \alpha'(u) \prec' (y \prec_l w + y \succ_l w + v \prec_r z + v \succ_r z + v \prec' w + v \succ' w) \\
&=(0 ,\ (x \prec y) \prec_{\ell} \alpha'(w) + (x \prec_{\ell} v + u \prec_{r} y + u \prec' v) \prec_{r} \alpha(z) 
 + (x \prec_{\ell} v + u \prec_{r} y) \prec' \alpha'(w)\\
 &- \alpha(x) \prec_{\ell} (y \prec_{\ell} w + y \succ_{\ell} w + v \prec_{r} z + v \succ_{r} z + v \prec' w + v \succ' w) - \alpha'(u) \prec_{r} (y \prec z + y \succ z)-\\
 &\alpha'(u) \prec'(y \prec_{\ell} w + y \succ_{\ell} w + v \prec_{r} z + v \succ_{r} z)) \\
 &((x,\ u ) \succ_{\bowtie} (y,\  v) ) \succ_{\bowtie}(\alpha(z) ,\ \alpha'(w)) - (\alpha(x) ,\ \alpha'(u))\succ_{\bowtie} ((y,\  v)  \succ _{\bowtie}(z ,\ w) + (y,\  v)  \succ _{\bowtie}(z ,\ w))\\
&=( (x \succ y) \prec \alpha(z)- \alpha(x) \succ (y \prec z) ,\ (x \succ y) \prec_l \alpha'(w) 
+ (x \succ_l v + u \succ_r y + u \succ' v) \prec_r \alpha(z)\\
&+ (x \succ_l v + u \succ_r y + u \succ' v) \prec' \alpha'(w)- \alpha(x) \succ_l (y \prec_l w+ v \prec_r z +  v \prec' w )\\
&- \alpha'(u) \succ_r (y \prec z )- \alpha'(u) \succ ' (y \prec_l w  + v \prec_r z + v \succ_r z + v \prec' w )) \\
&=(0 ,\ (x \succ y) \prec_{\ell} \alpha'(w) + (x \succ_l v + u \succ_r y + u \succ' v) \prec_r \alpha(z)\\ 
 &+(x \succ_l v + u \succ_r y + u \succ' v) \prec' \alpha'(w)-\alpha(x) \succ_l (y \prec_l w+ v \prec_r z +  v \prec' w )\\
&- \alpha'(u) \succ_r (y \prec z )- \alpha'(u) \succ ' (y \prec_l w  + v \prec_r z))\\
& (\alpha(x),\ \alpha'(u))\succ_{\bowtie} ((y,\  v)  \succ_{\bowtie} (z ,\ w))-((x,\  u) \prec_{\bowtie} (y,\  v)  + (x  \ u) \succ_{\bowtie} (y,\  v) ) \succ_{\bowtie} (\alpha(z)  ,\ \alpha'(w))\\
&=(\alpha(x) \succ(y\succ z)-(x\prec y+ x \succ y)\succ\alpha(z) ,\ \alpha(x) \succ_{l} (y \succ_l w+ v \succ_r z +  v \succ^{'} w )+\alpha'(u) \succ_r (y \succ z )\\
&+\alpha'(u) \succ ' (y \succ_l w  + v \succ_r z+v \succ'w)-(x\prec y )\succ_l\alpha'(w)-(x \prec_l v + u \prec_r y + u \prec' v)\succ_r\alpha(z)\\
&-(x \prec_l v + u \prec_r y + u \prec' v)\succ'\alpha'(w)
-(x\succ y )\succ_l\alpha'(w)-(x \succ_l v + u \succ_r y + u \succ'v)\succ_{r}\alpha(z)\\
&-(x \succ_l v + u \succ_r y + u \succ' v)\succ'\alpha'(w))\\
&=(0  ,\ \alpha(x)\succ_l (y \succ_l w+ v \succ_r z + v \succ' w)+\alpha'(u) \succ_r (y \succ z )+\alpha'(u) \succ'(y \succ_l w + v \succ_r z)\\
&-(x \prec y) \succ_l \alpha'(w)-(x \prec_l v + u \prec_r y + u \prec'v) \succ_r \alpha(z)-(x \prec_l v + u \prec_r y) \succ^{'}\alpha'(w)\\
 &- (x \succ y) \succ_l \alpha'(w)-(x \succ_l v + u \succ_r y + u \succ' v) \succ_r \alpha(z)-(x \succ_l v + u \succ_r y) \succ'\alpha'(w)).
\end{align*}
Then $(D \oplus D',\  \succ_{\bowtie} ,\ \prec_{\bowtie},\  \alpha\otimes \alpha')$ is a Hom-dendriform algebra if and only if $(D',\  \prec_l,\  \succ_l,\  \prec_r ,\ \succ_r,\  \alpha')$ is a representation of $D$ and Eqs.(\ref{eqtt})-(\ref{eqttc}) are satisfied.
\end{proof}
\begin{definition}\label{d13}
A linear map $T :\ D' \to D$ is called a homomorphic relative averaging operator on the Hom-dendriform algebra $(D ,\ \prec,\  \succ ,\ \alpha) $ with respect to an action $(D' ,\ \prec' ,\ \succ' ,\ \prec_l ,\ \succ_{l} ,\ \prec_r ,\  \succ_r,\  \alpha')$ if $T $ qualifies both as a relative averaging operator and a Hom-dendriform homomorphism.\\
A homomorphic relative averaging operator on a Hom-dendriform algebra is also called a relative averaging operator of weight 1 on a Hom-dendriform algebra. In particular  a homomorphic relative averaging operator on a Hom-dendriform algebra $(D,\  \prec ,\ \succ,\ \alpha)$ with respect to the adjoint representation is called a homomorphic averaging operator.
\end{definition}
\begin{definition}\label{d14}
A Hom-six-dendriform algebra is a tuple $(D, \prec_\perp,\succ_\perp,\prec_\vdash, \succ_\vdash,\prec_\dashv,\succ_\dashv,\alpha)$
consisting of a  Hom-dendriform algebra $(D,\prec_\perp,\succ_\perp,\alpha)$ and a Hom-quadri-dendriform algebra  $(D,\prec_\vdash, \succ_\vdash,\prec_\dashv,\succ_\dashv, \alpha)$
such that  for all $ x,\ y,\  z \in D $
\begin{align}
(x \prec_\vdash y) \prec_\perp \alpha(z) &= \alpha(x) \prec_\vdash (y \prec_\perp z + y \succ_\perp z) \label{sq1}  \\
(x \succ_\vdash y) \prec_\perp \alpha(z) &= \alpha(x) \succ_\vdash (y \prec_\perp z) \label{sq2}\\
\alpha(x) \succ_\vdash (y \succ_\perp z) &= (x \prec_\vdash y + x \succ_\vdash y) \succ_\perp\alpha(z) \label{sq3}\\
(x \prec_\dashv y) \prec_\perp \alpha( z) &= \alpha(x) \prec_\perp (y \prec_\vdash z + y \succ_\vdash z) \label{sq4}\\
(x \succ_\dashv y) \prec_\perp \alpha(z) &= \alpha(x)\succ_\perp (y \prec_\vdash z) \label{sq5}\\
\alpha(x)\succ_\perp (y \succ_\vdash z) &= (x \prec_\dashv y + x \succ_\dashv y) \succ_\perp \alpha(z) \label{sq6}\\
(x \prec_\perp y) \prec_\dashv \alpha(z) &= \alpha(x) \prec_\perp (y \prec_\dashv z + y \succ_\dashv z) \label{sq7} \\
(x \succ_\perp y) \prec_\dashv \alpha(z) &= \alpha(x) \succ_\perp (y \prec_\dashv z) \label{sq8}\\
\alpha(x)\succ_\perp (y \succ_\dashv z) &= (x \prec_\perp y + x \succ_\perp y) \succ_\dashv \alpha(z) \label{sq9}\\
(x \prec_\perp y) \prec_\vdash \alpha(z) &= (x \prec_\vdash y) \prec_\vdash \alpha(z) = (x \prec_\dashv y) \prec_\vdash \alpha(z) \label{sq10}\\
(x \succ_\perp y) \prec_\vdash\alpha(z) &= (x \succ_\vdash y) \prec_\vdash \alpha(z) = (x \succ_\dashv y) \prec_\vdash \alpha(z) \label{sq11}\\
(x \prec_\perp y) \succ_\vdash\alpha(z) &= (x \prec_\vdash y) \succ_\vdash \alpha( z) = (x \prec_\dashv y) \succ_\vdash\alpha(z) \label{sq12}\\
(x \succ_\perp y) \succ_\vdash \alpha(z) &= (x \succ_\vdash y) \succ_\vdash\alpha(z) = (x \succ_\dashv y) \succ_\vdash\alpha(z) \label{sq13}\\
\alpha(x) \prec_\dashv (y \prec_\perp z) &=\alpha(x) \prec_\dashv (y \prec_\vdash z) = \alpha(x) \prec_\dashv (y \prec_\dashv z) \label{sq14}\\
\alpha(x) \succ_\dashv (y \prec_\perp z) &=\alpha(x) \prec_\dashv (y \prec_\vdash z) = \alpha(x) \prec_\dashv (y \prec_\dashv z) \label{sq15}\\
\alpha(x) \succ_\dashv (y \succ_\perp z) &= \alpha(x) \succ_\dashv (y \succ_\vdash z) = \alpha(x) \succ_\dashv (y \succ_\dashv z) \label{sq16}\\
\alpha(x) \prec_\dashv (y \succ_\perp z) &= \alpha(x) \prec_\dashv (y \succ_\vdash z) = \alpha(x) \prec_\dashv (y \succ_\dashv z).\label{sq17}
\end{align}
$(D,\  \prec_\perp,\  \succ_\perp  ,\ \prec_\vdash  ,\ \succ_\vdash,\  \prec_\dashv,\  \succ_\dashv,\  \alpha)$ is said to be a multiplicative Hom-algebra if:
\begin{align*}
 \alpha(x\prec_\perp y)&=\alpha(x)\prec_\perp\alpha(y) ,\quad \alpha(x\succ_\perp y)=\alpha(x)\succ_\perp\alpha(y),\\
 \alpha(x\prec_\vdash y)&=\alpha(x)\prec_\vdash\alpha(y),
 \quad
\alpha(x\succ_\vdash y)=\alpha(x)\succ_\vdash\alpha(y),\\
\alpha(x\prec_\dashv y)&=\alpha(x)\prec_\dashv\alpha(y) ,\quad \alpha(x\succ_\dashv y)=\alpha(x)\succ_\dashv\alpha(y).
\end{align*}
If
$\prec_\perp = \prec_\vdash = \prec_\dashv \ and \ \succ_\perp = \succ_\vdash = \succ_\dashv $
then we simply obtain a Hom-dendriform algebra.\\ If
$\prec_\perp = \prec_\vdash = \prec_\dashv \ and \ \succ_\perp = \succ_\vdash = \succ_\dashv $
then we simply obtain a Hom-dendriform algebra.
\end{definition}
\begin{definition}\label{d15}
Given two Hom-six-dendriform algebras $(D,\  \prec_\perp,\  \succ_\perp  ,\ \prec_\vdash  ,\ \succ_\vdash,\  \prec_\dashv,\  \succ_\dashv,\  \alpha)$ and $(D' ,\ \prec_\perp' ,\ \succ_\perp',\  \prec_\vdash',\  \succ_\vdash',\  \prec_\dashv',\  \succ_\dashv'  ,\alpha')$  a Hom-six-dendriform homomorphism is a linear map $T: \ D \to D'$ that satisfies the following compatibility conditions:
 \begin{align*}
 T(x\prec_\perp y)&=T(x)\prec_\perp'T(y) ,\quad T(x\succ_\perp y)=T(x)\succ_\perp'T(y),\\
 T(x\prec_\vdash y)&=T(x)\prec_\vdash'T(y) ,\quad
T(x\succ_\vdash y)=T(x)\succ_\vdash' T(y),\\  T(x\prec_\dashv y)&=T(x)\prec_\dashv'T(y),\quad T(x\succ_\dashv y)=T(x)\succ_\dashv'T(y).
\end{align*}
In particular  if $T$ is bijective  then $(D,\  \prec_\perp,\  \succ_\perp  ,\ \prec_\vdash  ,\ \succ_\vdash,\  \prec_\dashv,\  \succ_\dashv,\  \alpha)$ and\\ $(D',\  \prec_\perp' ,\ \succ_\perp',\  \prec_\vdash' ,\ \succ_\vdash' ,\ \prec_\dashv' ,\ \succ_\dashv',\  \alpha')$ are said to be isomorphic.
\end{definition}
\begin{definition}\label{d16} 
A Hom-triassociative algebra is a vector space $D$ equipped with three binary operations
$
\vdash,\  \dashv ,\ \perp : \ D \otimes D \to D $ a linear map $\alpha : \ D \to D$ such that $(D,\  \vdash,\  \dashv,\  \alpha)$ constitutes a Hom-diassociative algebra  $(D,\  \perp,\  \alpha)$ forms an Hom-associative algebra  and for all $x,\ y,\  z \in D$  the following equations hold:
\begin{align}
(x \dashv y) \dashv \alpha(z)&= \alpha(x) \dashv (y \perp z)  \\
(x \vdash  y) \perp\alpha(z)&= \alpha(x) \vdash (y \perp z)  \\
(x \perp y) \dashv \alpha(z)& = \alpha(x) \perp (y \dashv z)  \\
(x \perp y) \vdash \alpha(z)&= \alpha(x) \vdash (y \vdash z)  \\
(x \dashv y) \perp \alpha(z)&= \alpha(x) \perp (y \vdash z).
\end{align}
 $(D ,\ \vdash,\  \dashv,\  \perp,\ \alpha)$ is said to be a multiplicative Hom-algebra if $\forall \ x,\ y \in D$:\\
$\alpha(x\vdash y)=\alpha(x)\vdash\alpha(y) ,\quad\alpha(x\dashv y)=\alpha(x)\dashv\alpha(y) \ and\ \alpha(x\perp y)=\alpha(x)\perp\alpha(y).$
\end{definition} 
\begin{proposition}
Let $(D ,\ \prec_\perp,\  \succ_\perp,\  \prec_\vdash,\  \prec_\dashv ,\ \succ_\vdash,\  \succ_\dashv ,\ \alpha)$ be a Hom-six-dendriform algebra.\\ Then $(D,\  \perp,\  \vdash,\  \dashv ,\ \alpha)$ is a Hom-triassociative algebra  where the operations are defined by:
\begin{align*}
x \perp y &= x \prec_\perp y + x \succ_\perp y ,\\
x \vdash y &= x \prec_\vdash y + x \succ_\vdash y ,\\
x \dashv y &= x \prec_\dashv y + x \succ_\dashv y ,\\
\end{align*}
for all $x,\ y \in D$.
\end{proposition}
\begin{proof}
It follows from Proposition \ref{qpro}  that $(D,\  \vdash,\  \dashv,\  \alpha)$ is a Hom-di-associative algebra. The fact
 that $(D ,\ \prec_\perp ,\  \succ_\perp,\  \alpha)$ is a Hom-dendriform algebra means that $(D,\  \perp,\  \alpha)$ is a Hom-associative algebra. Next  for any $x,\ y,\  z \in D$  
by Eqs.(\ref{Hq1})-(\ref{Hq5}) and (\ref{Hq6})-(\ref{Hq11})  we have
 \begin{align*}
(x \dashv y) \dashv \alpha(z)
&= (x \prec_\dashv y + x \succ_\dashv y) \prec_\dashv \alpha(z) + (x \prec_\dashv y + x \succ_\dashv y) \succ_\dashv \alpha(z) \\
&= \alpha(x) \prec_\dashv (y \prec_\vdash z + y \succ_\vdash z) + \alpha(x) \succ_\dashv (y \prec_\dashv z) + \alpha(x) \succ_\dashv (y \succ_\vdash z) \\
&= \alpha(x) \prec_\dashv (y \prec_\perp z + y \succ_\perp z) + \alpha(x) \succ_\dashv (y \prec_\perp z + y \succ_\perp z) \\
&= \alpha(x) \dashv (y \perp z) \\
\\
(x \perp y) \dashv \alpha(z)
&= (x \prec_\perp y + x \succ_\perp y) \prec_\dashv \alpha(z) + (x \prec_\perp y + x \succ_\perp y) \succ_\dashv \alpha(z) \\
&= \alpha(x) \prec_\perp (y \prec_\dashv z + y \succ_\dashv z) + \alpha(x) \succ_\perp (y \prec_\dashv z + y \succ_\dashv z) \\
&= \alpha(x) \perp (y \dashv z) \\
\\
(x \dashv y) \perp \alpha(z)
&= (x \prec_\dashv y + x \succ_\dashv y) \prec_\perp \alpha(z) + (x \prec_\dashv y + x \succ_\dashv y) \succ_\perp \alpha(z) \\
&= \alpha(x) \prec_\perp (y \prec_\vdash z + y \succ_\vdash z) + \alpha(x) \succ_\perp (y \prec_\vdash z + y \succ_\vdash z) \\
&= \alpha(x) \perp (y \vdash z) \\
\\
(x \vdash y) \vdash \alpha(z)
&= (x \prec_\vdash y + x \succ_\vdash y) \prec_\vdash \alpha(z) + (x \prec_\vdash y + x \succ_\vdash y) \succ_\vdash \alpha(z) \\
&= \alpha(x) \prec_\vdash (y \prec_\vdash z + y \succ_\vdash z) + \alpha(x) \succ_\vdash (y \prec_\vdash z + y \succ_\vdash z) \\
&= \alpha(x) \vdash (y \vdash z) \\
\\
(x \perp y) \vdash \alpha(z)
&= (x \prec_\perp y + x \succ_\perp y) \prec_\vdash \alpha(z) + (x \prec_\perp y + x \succ_\perp y) \succ_\vdash \alpha(z) \\
&= \alpha(x) \prec_\vdash (y \prec_\vdash z + y \succ_\vdash z) + \alpha(x) \succ_\vdash (y \prec_\vdash z + y \succ_\vdash z) \\
&= \alpha(x) \vdash (y \vdash z)
\end{align*}
This completes the proof.
\end{proof}
\begin{thm}
Let $T :\ D' \to D$ be a homomorphic relative averaging operator on the Hom-dendriform algebra $(D  ,\ \prec  ,\ \succ  ,\alpha)$ with respect to an action $(D',\  \prec' ,\ \succ',\  \prec_\ell ,\ \succ_\ell,\  \prec_r ,\ \succ_r ,\ \alpha')$. Then \\$(D',\  \prec_\perp,\  \succ_\perp,\  \prec^T_\vdash,\  \prec^T_\dashv ,\ \succ^T_\vdash ,\ \succ^T_\dashv,\  \alpha')$is a Hom-six-dendriform algebra, where
\begin{align*}
u \prec_\perp v &= u \prec' v ,\quad 
u \succ_\perp v = u \succ' v ,\\ 
u \prec^T_\vdash v &= T u \prec_l v ,\quad
u \prec^T_\dashv v = u \prec_r T v ,\\
u \succ^T_\vdash v &= T u \succ_l v ,\quad 
u \succ^T_\dashv v = u \succ_r T v 
\end{align*}
\end{thm} 
\begin{proof}
We have already $(D',\  \prec_\perp,\  \succ_\perp,\  \alpha')$ is a Hom-dendriform algebra. Moreover  it following from Proposition \ref{prop3.11} that $(D',\  \prec^{T}_\vdash,\   \prec^T_\dashv,\  \succ^T_\vdash,\  \succ^T_\dashv,\  \alpha')$ 
is a Hom-quadri-dendriform algebra. Now  for any $u,\  v,\  w\in D'$ according to Eqs.(\ref{eqtt})-(\ref{eqttc}), we have
\begin{align*}
(u \prec^T_{\vdash} v) \prec_{\perp} \alpha'(w)
&= (Tu \prec_{\ell} v) \prec' \alpha'(w) 
= \alpha (Tu) \prec_{\ell}(v \prec_{\perp} w + v \succ_{\perp} w) \\
&= T\alpha'(u) \prec_{\ell} (v \prec_{\perp} w + v \succ_{\perp} w) 
= \alpha'(u) \prec^T_{\vdash} (v \prec_{\perp} w + v \succ_{\perp} w) \\
(u \succ^T_{\vdash} v) \prec_{\perp} \alpha'(w)
&= (Tu \succ_{\ell} v) \prec' \alpha'(w) 
= \alpha(Tu) \succ_{\ell} (v \prec' w) \\
&= T\alpha'(u) \succ_{\ell} (v \prec' w) 
= \alpha'(u) \succ^T_{\vdash} (v \prec' w) 
= \alpha'(u) \succ^T_{\vdash}(v \succ_{\perp} w) \\
\alpha'(u) \succ^T_{\vdash}(v \succ_{\perp} w)
&= T\alpha'(u) \succ_{\ell} (v \prec' w) 
= (Tu \prec_{\ell} v + Tu \succ_{\ell} v) \succ' \alpha'(w) \\
&= (u \succ^T_{\vdash} v + u \prec^T_{\vdash} v) \succ_{\perp} \alpha'(w) \\
(u \prec^T_{\dashv} v) \prec_{\perp} \alpha'(w) 
&= (u \prec_{r} Tv) \prec' \alpha'(w) 
= \alpha'(u) \prec' (Tv \prec_{\ell} w + Tv \succ_{\ell} w) \\
&= \alpha'(u) \prec_{\perp} (v \prec^T_{\vdash} w + v \succ^T_{\vdash} w) \\
(u \succ^T_{\dashv} v) \prec_{\perp} \alpha'(w) 
&= (u \succ_{r} Tv) \prec' \alpha'(w) 
= \alpha'(u) \succ' (Tv \prec_{\ell} w) 
= \alpha'(u) \succ_{\perp} (v \prec^T_{\vdash} w) \\
\alpha'(u) \succ_{\perp} (v \succ^T_{\vdash} w) 
&= \alpha'(u) \succ' (Tv \succ_{\ell} w) 
= (u \prec_{r} Tv + u \succ_{r} Tv) \succ' \alpha'(w) \\
&= (u \prec^T_{\vdash} v + u \succ^T_{\vdash} v) \succ_{\perp} \alpha'(w) \\
(u \prec_{\perp} v) \prec^T_{\vdash} \alpha'(w) 
&= (u \prec' v) \prec_{r} T\alpha'(w) 
= (u \prec' v) \prec_{r} \alpha(Tw) \\
&= \alpha'(u) \prec' (v \prec_{r} Tw + v \succ_{r} Tw) 
= \alpha'(u) \prec_{\perp} (v \prec^T_{\vdash} w + v \succ^T_{\vdash} w) \\
(u \succ_{\perp} v) \prec^T_{\vdash} \alpha'(w) 
&= (u \succ' v) \prec_{r} T\alpha'(w) 
= (u \succ' v) \prec_{r} \alpha(Tw) \\
&= \alpha'(u) \succ' (v \prec_{r} Tw) 
= \alpha'(u) \succ_{\perp} (v \prec^T_{\vdash} w)
\end{align*}

Similarly it can be proved that other equations  \ref{sq8} à \ref{sq17} are also true. This completes the proof.
\end{proof}
\begin{cor}
Let $(D ,\ \prec,\  \succ ,\ \alpha)$ be a Hom-dendriform algebra and $T :\ D \to D$ be a homomorphic averaging operator. Then $(D' ,\ \prec_\perp,\  \succ_\perp,\  \prec^{T}_\vdash ,\ \prec^T_\dashv,\  \succ^T_\vdash,\  \succ^T_\dashv,\  \alpha)$ is a Hom-six-dendriform algebra  where \\
$\prec_\perp = \prec ,\quad \succ_\perp = \succ ,\quad 
x \prec^T_\vdash y = T x \prec y ,\quad 
x \prec^T_\dashv y = x \prec T y ,\quad 
x \succ^T_\vdash y = T x \succ y ,\quad
x \succ^T_\dashv y = x \succ T y $\\
for all $x,\ y \in D.$
\end{cor}
\begin{thm}
 Every Hom-six-dendriform algebra can induce a homomorphic relative averaging operator on a Hom-dendriform algebra with respect to a action.    
\end{thm}
\begin{proof}
  Let $(D,\  \prec_{\perp} ,\ \succ_{\perp},\  \prec_\vdash ,\ \prec_\dashv ,\ \succ_\vdash,\  \succ_\dashv,\  \alpha)$ be a Hom-six-dendriform algebra. The vector subspace \(I_D\) generated by
\[
\left\{
x \prec_\dashv y - x \prec_\vdash y  ,\quad
x \succ_\dashv y - x \succ_\vdash y 
\;\middle|\; x,\ y \in D
\right\}
\]
is an ideal of \(D\). The quotient algebra 
$(D / I_D,\  \prec,\   \succ ,\ \overline{\alpha})$
is a Hom-dendriform algebra  denoted by \(D_{Dend}\)  where
\[
\overline{\alpha}(\overline{x}) = \overline{\alpha(x)}  ,\quad
\overline{x} \prec \overline{y} = \overline{x \prec_{\vdash} y} = \overline{x \prec_{\dashv} y}  ,\quad
\overline{x} \succ \overline{y} = \overline{x \succ_{\vdash} y} = \overline{x \succ_{\dashv} y} 
\]
for all $\overline{x} ,\ \overline{y} \in D / I_D.$
On the other hand  \((D  ,\ \prec_{\perp} ,\ \succ_{\perp},\  \alpha)\) is a Hom-dendriform algebra  denoted by \(D_{\perp}\). \\ $D_{\perp}$ is a \(D_{Dend}\)-module defined by the actions:
\begin{align*}
\overline{x} \prec_l \overline{y}& = x \prec_{\vdash} y  ,\quad
\overline{x} \succ_l y = x \succ_{\vdash} y ,\\
y \prec_r \overline{x} &= y \prec_{\dashv} x  ,\quad
y \succ_r \overline{x} = y \succ_{\dashv} x.
\end{align*}
for all \(x,\ y \in D_{\perp}\). Moreover  the quotient map $T :\ D_{\perp} \to D_{Dend}  \quad x \mapsto \overline{x}$  is a homomorphic relative averaging operator.  
\end{proof}
\section{Classification}
In this section  we give the classification of Hom-quadri-dendriform algebra in low dimension.

For a selection of cases  we classify the Hom-quadri-dendriform algebras. The computations for our classifications were done using Mathematica. Throughout  we work over the complex field.
\begin{thm}
The isomorphism class of 2-dimensional Hom-quadri-dendriform algebras is given by the following representatives.
\begin{itemize}
\item
$
D_1:\left\{
    \begin{array}{ll}
      e_1\prec_\vdash e_1=e_2& e_1\succ_\vdash e_1=e_2\\
     e_1\prec_\dashv e_1=e_2&e_1\succ_\dashv e_1=\frac{1}{2}e_2
    \end{array}
\right.\quad
\alpha = 
\begin{pmatrix}
a & 1 \\
0 & a 
\end{pmatrix}; a\in\mathbb{R}.
$
\item
$D_2:\left\{
    \begin{array}{ll}
      e_1\prec_\vdash e_1=e_2& e_1\succ_\vdash e_1=-e_2\\
     e_1\prec_\dashv e_1=e_2&e_1\succ_\dashv e_1=e_2
    \end{array}
\right.\quad
\alpha = 
\begin{pmatrix}
a & 0 \\
0 & 0 
\end{pmatrix}; a\in\mathbb{R}.
$
\item
$D_3:\left\{
    \begin{array}{ll}
      e_1\prec_\vdash e_1=e_1& e_1\succ_\vdash e_1=e_1\\
			  e_2\prec_\vdash e_1=e_1& e_1\succ_\vdash e_1=e_1\\
     e_1\prec_\dashv e_1=e_1&e_1\succ_\dashv e_1=e_1\\
		e_2\prec_\dashv e_2=e_1&e_2\succ_\dashv e_2=e_1
    \end{array}
\right.\quad
\alpha = 
\begin{pmatrix}
0 & 0 \\
0 & b 
\end{pmatrix}; b\in\mathbb{R}.
$
\item
$D_4(\gamma):\left\{
    \begin{array}{ll}
      e_1\prec_\vdash e_2=e_1& e_1\succ_\vdash e_2=e_1\\
	e_2\prec_\vdash e_2=e_1& e_2\succ_\vdash e_2=e_1\\
     e_2\prec_\dashv e_2=\gamma e_2
    \end{array}
\right.\quad
\alpha = 
\begin{pmatrix}
0 & 1 \\
0 & 0 
\end{pmatrix}; \gamma\in\mathbb{R}.
$
\item
$D_5(\eta):\left\{
    \begin{array}{ll}
      e_1\prec_\vdash e_1=e_1& e_1\succ_\vdash e_1=e_1\\
			  e_2\prec_\vdash e_1=e_1& e_2\succ_\vdash e_2=e_1\\
     e_1\prec_\dashv e_1=e_1&e_1\succ_\dashv e_1=e_1\\
		e_2\prec_\dashv e_2=e_1&e_2\succ_\dashv e_2=\eta e_1
    \end{array}
\right.\quad
\alpha = 
\begin{pmatrix}
0 & 1 \\
0 & 0 
\end{pmatrix}; \eta\in\mathbb{R}.
$
\end{itemize}
\end{thm}

\begin{thm}
The isomorphism class of 3-dimensional Hom-quadri-dendriform algebras is given by the following representatives.
\begin{itemize}
\item
$
D_1:\left\{
    \begin{array}{llllll}
      e_1\prec_\vdash e_1=e_2& e_1\succ_\vdash e_1=e_2\\
			e_1\prec_\vdash e_3=e_2& e_1\succ_\vdash  e_3=e_2\\
			e_3\prec_\vdash e_3=e_2& e_3\succ_\vdash  e_3=e_2\\
			e_1\prec_\dashv e_1=e_2& e_1\succ_\dashv e_3=e_2\\
			e_1\prec_\dashv e_3=e_2& e_3\succ_\dashv  e_1=e_2\\
     e_3\prec_\dashv e_3=e_2&e_3\succ_\dashv e_3=e_2
    \end{array}
\right.\quad
\alpha = 
\begin{pmatrix}
a & 1 &0\\
0 &a &0 \\
0 &0 &a 
\end{pmatrix}; a\in\mathbb{R}.
$
\item
$D_2:\left\{
    \begin{array}{llllll}
      e_1\prec_\vdash e_1=e_3& e_1\succ_\vdash e_1=e_3\\
			e_1\prec_\vdash e_2=e_3& e_1\succ_\vdash e_2=e_3\\
			e_2\prec_\vdash e_1=e_3& e_2\succ_\vdash e_1=e_3\\
			e_2\prec_\vdash e_2=e_3& e_2\succ_\vdash e_2=e_3\\
			e_1\prec_\dashv e_1=e_3& e_1\succ_\dashv e_1=e_3\\
     e_1\prec_\dashv e_2=e_3&e_1\succ_\dashv e_2=e_3\\
		e_2\prec_\dashv e_1=e_3&e_2\succ_\dashv e_1=e_3\\
		e_2\prec_\dashv e_2=e_3&e_2\succ_\dashv e_2=e_3\\
    \end{array}
\right.\quad
\alpha = 
\begin{pmatrix}
a & 1 &0\\
0 &a &0 \\
0 &0 &a 
\end{pmatrix}; a\in\mathbb{R}.
$
\item
$D_3 :\left\{
    \begin{array}{llllll}
      e_1\prec_\vdash e_2=e_1+e_3& e_1\succ_\vdash e_2=e_1+e_3\\
			e_2\prec_\vdash e_1=e_1+e_3& e_2\succ_\vdash e_1=e_1+e_3\\
			e_2\prec_\vdash e_3=e_1+e_3& e_2\succ_\vdash e_2=\sigma e_3\\
			e_3\prec_\vdash e_3=e_1+e_3& e_2\succ_\vdash e_3=e_1+e_3\\
			e_1\prec_\dashv e_2=e_3& e_3\succ_\vdash e_2=e_1+e_3\\
     e_2\prec_\dashv e_1=e_3&e_3\succ_\vdash e_3=e_3\\
		e_2\prec_\dashv e_2=e_3&e_2\succ_\dashv e_1=e_3\\
		e_3\prec_\dashv e_2=e_3&e_2\succ_\dashv e_2=e_3\\
		e_3\prec_\dashv e_3=e_3&e_2\succ_\dashv e_3=e_3\\
	&e_3\succ_\dashv e_2=-\eta e_3\\
    \end{array}
\right.\quad
\alpha = 
\begin{pmatrix}
0& 1 &0\\
0 &0 &0 \\
0 &0 &0 
\end{pmatrix}; \sigma, \eta\in\mathbb{R}.
$
\item
$D_4:\left\{
    \begin{array}{llllll}
      e_1\prec_\vdash e_3=e_1+e_2& e_1\succ_\vdash e_3=e_1+e_2\\
			e_2\prec_\vdash e_2=e_1+e_2& e_2\succ_\vdash e_2=e_1+e_2\\
			e_2\prec_\vdash e_3=e_1+e_2& e_2\succ_\vdash e_3=e_1+e_2\\
			e_3\prec_\vdash e_1=e_1+e_2& e_3\succ_\vdash e_1=e_1+e_2\\
			e_3\prec_\vdash e_3=e_1+e_2& e_3\succ_\vdash e_2=e_1+e_2\\
			e_1\prec_\dashv e_3=e_2& e_3\succ_\vdash e_3=e_1+e_2\\
     e_2\prec_\dashv e_2=e_2&e_2\succ_\dashv e_2=e_2\\
		e_2\prec_\dashv e_3=e_2&e_2\succ_\dashv e_3=e_2\\
		e_3\prec_\dashv e_1=e_2&e_3\succ_\dashv e_1=e_2\\
		e_3\prec_\dashv e_2=e_2&e_3\succ_\dashv e_2=e_2\\
		e_3\prec_\dashv e_3=e_2&e_3\succ_\dashv e_3=e_2
    \end{array}
\right.\quad
\alpha = 
\begin{pmatrix}
0& 1 &0\\
0 &0 &0 \\
0 &0 &0 
\end{pmatrix}; \sigma,\eta\in\mathbb{R}.
$
\item
$D_5:\left\{
    \begin{array}{llllll}
      e_1\prec_\vdash e_3=e_2& e_1\succ_\vdash e_3=e_2\\
			e_2\prec_\vdash e_2=e_2& e_2\succ_\vdash e_2=e_2\\
			e_3\prec_\vdash e_1=e_2& e_3\succ_\vdash e_1=e_2\\
			e_3\prec_\vdash e_3=e_2& e_3\succ_\vdash e_3=e_2\\
			e_1\prec_\dashv e_3=e_2& e_1\succ_\dashv  e_3=e_2\\
     e_2\prec_\dashv e_2=e_2&e_2\succ_\dashv e_2=e_2\\
		e_3\prec_\dashv e_1=e_2&e_3\succ_\dashv e_1=e_2\\
		e_3\prec_\dashv e_3=e_2&e_3\succ_\dashv e_3=e_2
    \end{array}
\right.\quad
\alpha = 
\begin{pmatrix}
0& 1 &0\\
0 &0 &0 \\
0 &0 &b 
\end{pmatrix}; b\in\mathbb{R}.
$
\item
$D_6:\left\{
    \begin{array}{llllll}
      e_1\prec_\vdash e_2=e_3& e_1\succ_\vdash e_2=e_3\\
			e_2\prec_\vdash e_1=e_3& e_2\succ_\vdash e_1=e_3\\
			e_2\prec_\vdash e_2=e_3& e_2\succ_\vdash e_2=e_3\\
			e_3\prec_\vdash e_3=e_3& e_3\succ_\vdash e_3=e_3\\
			e_1\prec_\dashv e_2=e_3& e_1\succ_\dashv  e_1=e_2\\
     e_2\prec_\dashv e_1=e_3&e_2\succ_\dashv e_1=e_3\\
		e_2\prec_\dashv e_2=e_3&e_2\succ_\dashv e_2=e_3\\
		e_3\prec_\dashv e_3=e_3&e_3\succ_\dashv e_3=e_3
    \end{array}
\right.\quad
\alpha = 
\begin{pmatrix}
0& 1 &0\\
0 &0 &1 \\
0 &0 &0
\end{pmatrix}.
$
\item
$D_7:\left\{
    \begin{array}{llllll}
      e_1\prec_\vdash e_1=e_3& e_1\succ_\vdash e_1=e_3\\
			e_1\prec_\vdash e_2=e_3& e_2\succ_\vdash e_1=e_3\\
			e_2\prec_\vdash e_1=e_3& e_2\succ_\vdash e_2=e_3\\
			e_2\prec_\vdash e_2=e_3& e_3\succ_\vdash e_3=e_3\\
			e_1\prec_\dashv e_1=e_3& e_1\succ_\dashv  e_1=e_3\\
     e_1\prec_\dashv e_2=e_3&e_1\succ_\dashv e_2=e_3\\
		e_2\prec_\dashv e_1=e_3&e_2\succ_\dashv e_1=e_3\\
		e_2\prec_\dashv e_2=e_3&e_2\succ_\dashv e_2=e_3
    \end{array}
\right.\quad
\alpha = 
\begin{pmatrix}
a& 1 &0\\
0 &a &1 \\
0 &0 &a 
\end{pmatrix}; a\in\mathbb{R}.
$

\item
$D_8:\left\{
    \begin{array}{llllll}
      e_1\prec_\vdash e_1=e_2& e_1\succ_\vdash e_1=e_2\\
			e_1\prec_\vdash e_3=e_2& e_1\succ_\vdash e_3=e_2\\
			e_3\prec_\vdash e_1=e_2& e_3\succ_\vdash e_1=e_2\\
			e_3\prec_\vdash e_3=e_2& e_3\succ_\vdash e_3=e_2\\
			e_1\prec_\dashv e_1=e_2& e_1\succ_\dashv  e_1=e_2\\
     e_1\prec_\dashv e_3=e_2&e_1\succ_\dashv e_3=e_2\\
		e_3\prec_\dashv e_1=e_2&e_3\succ_\dashv e_1=e_2\\
		e_3\prec_\dashv e_3=e_3&e_3\succ_\dashv e_3=e_2
    \end{array}
\right.\quad
\alpha = 
\begin{pmatrix}
1& 1 &0\\
0 &1 &1 \\
0 &0 &1 
\end{pmatrix}.
$
\item
$D_9:\left\{
    \begin{array}{llllll}
      e_1\prec_\vdash e_1=e_3& e_1\succ_\vdash e_1=e_3\\
			e_1\prec_\vdash e_2=e_3& e_1\succ_\vdash e_2=e_3\\
			e_2\prec_\vdash e_1=e_3& e_2\succ_\vdash e_2=e_3\\
			e_3\prec_\vdash e_3=e_3& e_3\succ_\vdash e_3=e_3\\
			e_1\prec_\dashv e_1=e_3& e_1\succ_\dashv  e_2=e_3\\
     e_2\prec_\dashv e_1=e_3&e_2\succ_\dashv e_1=e_3\\
		e_2\prec_\dashv e_2=e_3&e_2\succ_\dashv e_2=e_3\\
		e_3\prec_\dashv e_3=e_3&e_3\succ_\dashv e_3=e_3
    \end{array}
\right.\quad
\alpha = 
\begin{pmatrix}
a& 0 &0\\
0 &b &0 \\
0 &0 &0 
\end{pmatrix}.
$
\item
$D_{10}:\left\{
    \begin{array}{llllll}
      e_1\prec_\vdash e_3=e_1+e_2& e_1\succ_\vdash e_3=e_1+e_2\\
			e_3\prec_\vdash  e_1=e_1+\lambda e_2& e_3\succ_\vdash e_1=e_1+e_2\\
			e_3\prec_\vdash e_3=e_1+\frac{1}{2}e_2& e_3\succ_\vdash e_3=e_1+e_2\\
			e_1\prec_\dashv e_3=e_1+e_2& e_1\succ_\dashv  e_3=e_1+e_2\\
     e_3\prec_\dashv e_1=e_1+e_2&e_3\succ_\dashv e_1=e_1+e_2\\
		e_3\prec_\dashv e_3=e_1+e_2&e_3\succ_\dashv e_3=e_1+e_2
    \end{array}
\right.\quad
\alpha = 
\begin{pmatrix}
a& 0 &0\\
0 &b &0 \\
0 &0 &0 
\end{pmatrix}; a,  b, \lambda\in\mathbb{R}.
$

\item
$D_{11}:\left\{
    \begin{array}{llllll}
      e_1\prec_\vdash e_2=e_1+e_3& e_2\succ_\vdash e_1=e_1+e_3\\
			e_2\prec_\vdash  e_1=e_1+e_3& e_2\succ_\vdash e_2=e_1+e_3\\
			e_2\prec_\vdash e_2=e_1+e_3& e_3\succ_\vdash e_2=e_1+e_3\\
			e_3\prec_\vdash e_2=e_1+e_3& e_2\succ_\dashv e_1=e_1+e_3\\
     e_1\succ_\vdash e_2=e_1+e_3& e_2\succ_\dashv e_2=\eta e_1+e_3
    \end{array}
\right.\quad
\alpha = 
\begin{pmatrix}
a& 0 &0\\
0 &0 &0 \\
0 &0 &c 
\end{pmatrix}; \eta,  a,  c\in\mathbb{R}.
$

\item
$D_{12}:\left\{
    \begin{array}{llllll}
      e_1\prec_\vdash e_2=e_1& e_2\succ_\vdash e_1=e_1\\
			e_2\prec_\vdash  e_1=e_1& e_2\succ_\vdash e_2=e_1\\
			e_2\prec_\vdash e_2=e_1& e_2\succ_\vdash e_3=e_1\\
			e_2\prec_\vdash e_3=e_1& e_3\succ_\vdash e_3=e_1\\
     e_3\prec_\vdash e_3=e_1& e_2\succ_\dashv e_2=e_1\\
		e_1\succ_\vdash e_2=e_1& e_2\succ_\dashv e_3=\frac{1}{2}\eta e_1
    \end{array}
\right.\quad
\alpha = 
\begin{pmatrix}
1& 0 &0\\
0 &0 &0 \\
0 &0 &1 
\end{pmatrix}; \eta\in\mathbb{R}.
$

\item
$D_{13}:\left\{
    \begin{array}{llllll}
      e_1\prec_\vdash e_1=e_3& e_1\succ_\vdash e_1=e_3\\
			e_1\prec_\vdash  e_2=e_3& e_2\succ_\vdash e_1=e_3\\
			e_2\prec_\vdash e_1=e_3& e_2\succ_\vdash e_1=e_3\\
			e_2\prec_\vdash e_2=e_3& e_2\succ_\vdash e_2=e_3\\
     e_1\prec_\vdash e_1=e_3& e_1\succ_\dashv e_1=e_3\\
		 e_1\prec_\vdash e_2=e_3& e_1\succ_\dashv e_2=e_3\\
		 e_2\prec_\vdash e_1=e_3& e_2\succ_\dashv e_1=e_3\\
		 e_2\prec_\vdash e_2=e_3& e_2\succ_\dashv e_2=e_3
    \end{array}
\right.\quad
\alpha = 
\begin{pmatrix}
a& 0 &0\\
0 &b &0 \\
0 &0 &c 
\end{pmatrix};  b , c\in\mathbb{R}.
$
\end{itemize}
\end{thm}

\begin{proposition}
The following matrix forms characterize  averaging operators of $2$-dimensional  Hom-quadri-dendriform algebras:
\begin{enumerate}
	\item[$D_1$] : 
$\left(\begin{array}{ccc}
\vartheta_{11}&0\\
\vartheta_{21}&\vartheta_{11}
\end{array}
\right)$
$\left(\begin{array}{ccc}
0&0\\
\vartheta_{21}&\vartheta_{22}
\end{array}
\right)$ ;
\quad
$D_2$ : 
$\left(\begin{array}{ccc}
\vartheta_{11}&0\\
\vartheta_{21}&\vartheta_{11}
\end{array}
\right)$ 
$\left(\begin{array}{ccc}
0&0\\
\vartheta_{21}&\vartheta_{22}
\end{array}
\right)$
\item[$D_3$] : 
$\left(\begin{array}{ccc}
\vartheta_{22}&0\\
0&\vartheta_{22}
\end{array}
\right)$;
\quad
$D_4$ : 
$\left(\begin{array}{ccc}
\vartheta_{11}&0\\
0&\vartheta_{11}
\end{array}
\right)$
$\left(\begin{array}{ccc}
0&\vartheta_{12}\\
0&0
\end{array}
\right)$
;
\quad
$D_5$ : 
$\left(\begin{array}{ccc}
\vartheta_{22}&0\\
0&\vartheta_{22}
\end{array}
\right).$
\end{enumerate}
\end{proposition}

\begin{proposition}
The following matrix forms characterize  averaging operators of $3$-dimensional  Hom-quadri-dendriform algebras:
\begin{enumerate}
	\item[$D_1$] : 
$\left(\begin{array}{ccc}
\vartheta_{11}&0&0\\
0&\vartheta_{11}&0\\
0&0&\vartheta_{11}
\end{array}
\right)$
$\left(\begin{array}{ccc}
0&0&0\\
0&\vartheta_{22}&\vartheta_{23}\\
0&0&0
\end{array}
\right) 
\left(\begin{array}{ccc}
0&0&0\\
0&0&0\\
0&0&\vartheta_{23}
\end{array}
\right)$;\ 
$D_2$ : 
$\left(\begin{array}{ccc}
\vartheta_{22}&0&0\\
0&\vartheta_{22}&0\\
0&0&\vartheta_{22}
\end{array}
\right)$;
\item[$D_3$] : 
$\left(\begin{array}{ccc}
\vartheta_{33}&0&0\\
0&\vartheta_{33}&0\\
0&0&\vartheta_{33}
\end{array}
\right)$ 
$\left(\begin{array}{ccc}
0&\vartheta_{12}&0\\
0&0&0\\
0&0&0
\end{array}
\right);$\ 
$D_4$ : 
$\left(\begin{array}{ccc}
\vartheta_{22}&0&0\\
0&\vartheta_{22}&0\\
0&0&\vartheta_{22}
\end{array}
\right)$
$\left(\begin{array}{ccc}
0&0&\vartheta_{13}\\
0&0&0\\
0&0&0
\end{array}
\right)
;
$
\item[$D_5$] : 
$\left(\begin{array}{ccc}
\vartheta_{33}&0&0\\
0&\vartheta_{33}&0\\
0&0&\vartheta_{33}
\end{array}
\right)$
$\left(\begin{array}{ccc}
\vartheta_{11}&0&\vartheta_{13}\\
0&0&0\\
0&0&0
\end{array}
\right) $\ 
$\left(\begin{array}{ccc}
0&0&0\\
0&\vartheta_{22}&0\\
0&0&0
\end{array}
\right)$
;

\item[$D_6$] : 
$\left(\begin{array}{ccc}
\vartheta_{22}&0&0\\
0&\vartheta_{22}&0\\
0&0&\vartheta_{22}
\end{array}
\right)$
$\left(\begin{array}{ccc}
0&0&0\\
\vartheta_{21}&\vartheta_{22}&0\\
i\vartheta_{21}&i\vartheta_{22}&0\\
\end{array}
\right) $\ 
$\left(\begin{array}{ccc}
0&0&0\\
0&0&0\\
0&0&\vartheta_{33}
\end{array}
\right)$
;

\item[$D_7$] : 
$\left(\begin{array}{ccc}
\vartheta_{33}&0&0\\
0&\vartheta_{33}&0\\
0&0&\vartheta_{33}
\end{array}
\right)$
$\left(\begin{array}{ccc}
0&-\vartheta_{33}&0\\
0&\vartheta_{33}&0\\
0&\vartheta_{33}&0\\
\end{array}
\right) $\ 
$\left(\begin{array}{ccc}
0&0&0\\
0&0&0\\
0&0&\vartheta_{33}
\end{array}
\right)$
;

\item[$D_8$] : 
$\left(\begin{array}{ccc}
\vartheta_{11}&0&\vartheta_{13}\\
\vartheta_{21}&\vartheta_{22}&\vartheta_{23}\\
\vartheta_{22}-\vartheta_{11}&0&\vartheta_{22}-\vartheta_{13}
\end{array}
\right)$
$\left(\begin{array}{ccc}
\vartheta_{11}&-\vartheta_{12}&\vartheta_{13}\\
\vartheta_{21}&\vartheta_{22}&\vartheta_{23}\\
-\vartheta_{11}&-\vartheta_{12}&-\vartheta_{13}
\end{array}
\right)\ $
;
$D_9$ : 
$\left(\begin{array}{ccc}
\vartheta_{33}&0&0\\
0&\vartheta_{33}&0\\
0&0&\vartheta_{33}
\end{array}
\right)$;

\item[$D_{10}$] : 
$\left(\begin{array}{ccc}
\vartheta_{33}&0&0\\
0&\vartheta_{33}&\vartheta_{23}\\
0&0&\vartheta_{33}
\end{array}
\right)$
$\left(\begin{array}{ccc}
0&0&\vartheta_{13}\\
0&0&\vartheta_{23}\\
0&0&0\\
\end{array}
\right) $\ 
$\left(\begin{array}{ccc}
\vartheta_{33}&0&0\\
0&0&\vartheta_{23}\\
0&0&\vartheta_{33}
\end{array}
\right) $
$\left(\begin{array}{ccc}
0&0&0\\
\vartheta_{21}&\vartheta_{22}&\vartheta_{23}\\
0&0&0
\end{array}
\right)$
;

\item[$D_{11}$] : 
$\left(\begin{array}{ccc}
\vartheta_{22}&0&0\\
0&\vartheta_{22}&0\\
0&0&\vartheta_{22}
\end{array}
\right)$
$\left(\begin{array}{ccc}
0&0&0\\
0&0&0\\
0&\vartheta_{32}&0\\
\end{array}
\right) $\ 
$\left(\begin{array}{ccc}
0&-\vartheta_{22}&0\\
0&\vartheta_{22}&0\\
0&\vartheta_{22}&0
\end{array}
\right) $
$\left(\begin{array}{ccc}
0&0&0\\
0&0&0\\
-\vartheta_{33}&\vartheta_{32}&\vartheta_{33}
\end{array}
\right)$
;
\item[$D_{12}$] : 
$\left(\begin{array}{ccc}
\vartheta_{11}&0&0\\
0&\vartheta_{11}&0\\
0&0&\vartheta_{11}
\end{array}
\right)$
$\left(\begin{array}{ccc}
0&\vartheta_{12}&\vartheta_{13}\\
0&0&0\\
0&0&0\\
\end{array}
\right) ; $
\item[$D_{13}$] :  
$\left(\begin{array}{ccc}
\vartheta_{11}&\vartheta_{12}&\vartheta_{13}\\
-\vartheta_{11}&-\vartheta_{12}&-\vartheta_{31}\\
\vartheta_{31}&\vartheta_{32}&\vartheta_{33}
\end{array}
\right)$
$\left(\begin{array}{ccc}
\vartheta_{11}&\vartheta_{12}&0\\
\vartheta_{21}&\vartheta_{11}-\vartheta_{12}+\vartheta_{21}&0\\
\vartheta_{12}&\vartheta_{32}&\vartheta_{11}+\vartheta_{22}
\end{array}
\right)$.
\end{enumerate}
\end{proposition}

\section*{ Acknowledgement}
The authors thank the anonymous referees for their valuable suggestions and comments. 

%%%%%%%%%%%%%%%%%%%%%%%%%%%%%%%%%%%%%%%%%%%%%%%%%%%%%%%%%%%%%%%%%%%%%%%
\section*{ Conflicts of Interest}
%%%%%%%%%%%%%%%%%%%%%%%%%%%%%%%%%%%%%%%%%%%%%%%%%%%%%%%%%%%%%%%%%%%%%%%
The authors declare no conflicts of interest.
%%%%%%%%%%%%%%%%%%%%%%%%%%%%%%%%%%%%%%%%%%%%%%%%%%%%%%%%%%%%%%%%%%%%%%%%%%%%%%%%%%%%%%%%%%%%%%%%%%%%%%%%%%%%%%%%%%%%%%%%%%%%%%%%%%%%%%%%%%%%%%%%
 
\end{document}